\documentclass{amsart}
\usepackage{amsfonts,amsthm,latexsym,amsmath,amssymb,amscd,amsmath, mathrsfs, epsf}
\usepackage{graphicx}
\usepackage{float}
\usepackage{soul,xcolor}
\usepackage{etoolbox} 

\usepackage{pgf,tikz,graphics}
\usetikzlibrary{arrows}
\usetikzlibrary[patterns]
\usetikzlibrary{calc}
\usepackage{verbatim}
\usetikzlibrary{shapes}
\usetikzlibrary{decorations.pathmorphing}

\newtheorem{theorem}{Theorem}
\newtheorem{proposition}{Proposition}[section]
\newtheorem{lemma}[proposition]{Lemma}
\newtheorem{corollary}[proposition]{Corollary}

\newtheorem{conjecture}[proposition]{Conjecture}

\newtheorem{LEM}[proposition]{Lemma}

\newtheorem{CONJ}[proposition]{Conjecture}

\theoremstyle{definition}
\newtheorem{definition}[proposition]{Definition}
\newtheorem{example}[proposition]{Example}

\theoremstyle{remark}
\newtheorem{remark}[proposition]{Remark}

\AtEndEnvironment{example}{\null\hfill$\diamond$}%
\AtEndEnvironment{remark}{\null\hfill$\diamond$}%

\newcommand{\ZZ}{\mathbb{Z}}
\newcommand{\RR}{\mathbb{R}}
\newcommand{\Ss}{\mathbb{S}}
\newcommand{\fW}{\mathfrak{W}}
\newcommand{\Fl}{\operatorname{Fl}}
\newcommand{\Lo}{\operatorname{Lo}}
\newcommand{\Up}{\operatorname{Up}}
\newcommand{\GL}{\operatorname{GL}}
\newcommand{\SO}{\operatorname{SO}}
\newcommand{\Spin}{\operatorname{Spin}}
\newcommand{\swminor}{\operatorname{swminor}}
\newcommand{\mult}{\operatorname{mult}}
\newcommand{\Bru}{\operatorname{Bru}}
\newcommand{\transpose}{\top}
\newcommand{\frakl}{\mathfrak{l}}
\newcommand{\inv}{\operatorname{inv}}
\newcommand{\iti}{\operatorname{iti}}
\newcommand{\IdI}{\operatorname{Id}}
\newcommand{\Quat}{\operatorname{Quat}}
\newcommand{\Pos}{\operatorname{Pos}}
\newcommand{\Neg}{\operatorname{Neg}}
\newcommand{\Cont}{\operatorname{Cont}}
\newcommand{\rk}{\operatorname{rk}}
\newcommand{\bQ}{\mathbf{Q}}
\newcommand{\cC}{\mathcal{C}}
\newcommand{\cU}{\mathcal{U}}
\newcommand{\cL}{\mathcal{L}}
\newcommand{\cJ}{\mathcal{J}}

\newcommand\llbracket{[\![}
\newcommand\rrbracket{]\!]}

\newcommand{\diag}{\mathop\mathrm{diag}\nolimits}

\newcommand{\bR}{\mathbb{R}}
\newcommand{\ee}{\end{equation}}
\newcommand {\ga}{\gamma}
\newcommand {\Ga}{\Gamma}

\newcommand{\C}{\mathcal C}
\newcommand{\cH}{\mathcal H}

\newcommand \bL {\mathbf L}
\def \bP {\mathbb RP}

\def\Ff{{\mathcal F}}
\def\MyFig{0.6\textwidth}

\begin{document}
          \numberwithin{equation}{section}

          \title[Grassmann convexity and multiplicative Sturm theory, revisited]
          {Grassmann convexity and \\ multiplicative Sturm theory, revisited}

          \author[N.~Saldanha]{Nicolau Saldanha}
\address{ Departamento de Matem\'atica, PUC-Rio
R. Mq. de S. Vicente 225, Rio de Janeiro, RJ 22451-900, Brazil}
\email{ saldanha@puc-rio.br}

\author[B.~Shapiro]{Boris Shapiro}
\address{Department of Mathematics, Stockholm University, SE-106 91
Stockholm,
         Sweden}
\email{shapiro@math.su.se}

\author[M.~Shapiro]{Michael  Shapiro}
\address{Department of Mathematics, Michigan State University,
          East Lansing, MI 48824-1027}
\email{mshapiro@math.msu.edu}

\dedicatory{To  the late Vladimir Arnold, who started all this}
\date{}


\begin{abstract}
In this  paper we settle a special case of
the  Grassmann convexity conjecture formulated in \cite{ShSh}.
We present a conjectural formula  for the maximal  total number of real zeros
of the consecutive Wronskians  of an arbitrary fundamental solution to a disconjugate linear ordinary differential equation with real time
(compare with  \cite {ShShSur}).
We  show that this formula gives the lower bound for the required total number of real zeros for equations of an arbitrary order and,  using our results on the Grassmann convexity,  we prove that the aforementioned formula is correct  for equations  of orders $4$ and $5$.  
 \end{abstract}

\maketitle

\bigskip

\section{Introduction and main results} 

Our subject of study  is related to the PhD theses of the second and  third authors defended in the early 90s (see \cite {Sh, MShPhD}). Namely, the thesis of the second author contains  Conjecture~\ref{conj:Wr}, see below, but the presented argument  which is supposed to prove it  is false.  The statement itself is still open and (if proven) would be  of  fundamental importance to the general qualitative theory of linear ordinary differential equations with real time. The thesis of the third author contains  a number of  Schubert calculus problems relevant to Conjecture~\ref{conj:Wr}. Over the years the authors made several attempts to  
settle  it and,  in particular,  worked out  some reformulations and special cases.  This paper contains a number of new results in  that direction. (In what follows, we will label conjectures, theorems and lemmas borrowed from the existing literature by letters. Results  and conjectures labelled by numbers are new). 

\smallskip
We start with the following classical definition, see e.g. \cite{Co}.

\begin{definition}  
{A linear ordinary homogeneous differential equation 
\begin{equation}\label{eq:disc}
y^{(n)}+p_1(x)y^{(n-1)}+\ldots+p_n(x)y=0
\end{equation}
of order $n$ with real-valued continuous coefficients $p_i(x)$ 
defined on an interval $I \subseteq \bR$ is called \emph{disconjugate  on $I$}
if any of  its nontrivial solutions has at most $(n-1)$ zeros on $I$
counting multiplicities.
(The interval $I$ can either be open or closed).}
\end{definition}

\begin{CONJ}[Upper bound on the number of real zeros of a Wronskian]
\label{conj:Wr}
Given any equation \eqref{eq:disc} disconjugate on $I$, 
a positive integer $1\le k \le n-1$,
and an arbitrary $k$-tuple $(y_1(x),y_2(x),\dots, y_k(x))$
of its linearly independent solutions, the number of real zeros of
$\det(W(y_1(x),y_2(x),\dots, y_k(x)))$ on $I$
counting multiplicities does not exceed $k(n-k)$.
\end{CONJ}

Here 
$$W(y_1(x),y_2(x),\dots, y_k(x))=
\begin{pmatrix}  
y_1(x) & y'_1(x) & \dots &y^{(k-1)}_1(x) \\ 
y_2(x) & y'_2(x) & \dots &y_2^{(k-1)}(x)\\
\vdots & \vdots & \ddots & \vdots\\
y_k(x) & y_k'(x) & \dots &y_k^{(k-1)}(x)\
\end{pmatrix}$$
is the Wronskian matrix of the $k$-tuple  $(y_1(x),y_2(x),\dots, y_k(x))$. 

\smallskip

Cases $k=1$ and $k=n-1$ of Conjecture \ref{conj:Wr} are straightforward,
but not very illuminating.
The simplest non-trivial case $k=2,\; n=4$ of Conjecture~\ref{conj:Wr}
has been settled in \cite {ShSh}.
The first main result of the present paper extends these results.

\begin{theorem}\label{th:main} 
Conjecture~\ref{conj:Wr} holds for $k=2$ and $k=n-2$ for any $n\ge 3$.  
\end{theorem} 

The case $k = n-2$ follows easily from the case $k = 2$;
see Remark \ref{remark:star}.

\smallskip

We present an equivalent statement to Conjecture~\ref{conj:Wr} 
in a different language.
This version is used in this paper and is also in some references
(including \cite{GoSa0}).
Let $\Lo_n^1$ be the nilpotent Lie group of lower triangular $n\times n$
real matrices whose diagonal entries equal $1$.
Let $\cJ$ be the set of $n\times n$ real matrices $A$
such that the entry $A_{ij}$
is positive if $i = j+1$ and $0$ otherwise.
A smooth curve $\Gamma: [0,1] \to \Lo_n^1$ is {\em flag-convex}
(or, sometimes, just {\em convex})
if,  for all $t \in [0,1]$,  $(\Gamma(t))^{-1}\Gamma'(t) \in \cJ$.
For a flag-convex curve $\Gamma$ and an integer $k$, $0 < k < n$,
let 
\begin{equation}
\label{equation:mk}
m_k: [0,1] \to \RR, \qquad
m_k(t) = \det(\swminor(\Gamma(t),k)). 
\end{equation}
Here $\swminor(L,k)$ is the southwest $k\times k$ submatrix of $L$, 
i.e., the submatrix formed by its last $k$ rows and its first $k$ columns. 
(If the curve $\Gamma$ is not obvious from the context,
we write $m_{\Gamma,k} = m_k$.) 
Conjecture \ref{conj:Wr} above is equivalent to saying that
the number of zeroes $t \in [0,1]$
of the smooth function $m_k: [0,1] \to \RR$ is at most $k(n-k)$;
here, as above, zeroes are counted with multiplicities.
This equivalence will be additionally clarified in \S\ref{sec:La}.

\smallskip

Thus, Theorem \ref{th:main} claims that, 
for any flag-convex function $\Gamma: I \to \Lo_n^1$, 
the functions $m_2$ and $m_{n-2}$ have at most $2(n-2)$ zeroes each.
Our second result is an inequality in the opposite sense;
we state a related result in Corollary \ref{coro:ineq} below.


\begin{theorem}
\label{th:La}
Consider a smooth flag-convex curve $\Gamma_\bullet: I \to \Lo_n^1$
(where $I \subset \RR$ is a non-degenerate interval).
Then, for any open subinterval $I_1 \subset I$,
there exists a matrix $L_1  \in \Lo_n^1$ such that,
for $\Gamma_1(t) = L_1 \Gamma_\bullet(t)$ and $m_k = m_{\Gamma_1,k}$,
the following properties hold:
\begin{enumerate}
\item{all roots of each $m_k$ in $I$ are simple and
belong to $I_1$;}
\item{they are distinct:
if $k_1 \ne k_2$ then $m_{k_1}$ and $m_{k_2}$ have no common roots;}
\item{for each $k$, the function $m_k$ admits precisely $k(n-k)$
roots in $I$.}
\end{enumerate}
\end{theorem}

In this case the total number of roots
of all functions $m_k: I \to \RR$ is $\frac{n^3-n}{6}$.

\medskip

\medskip The structure of this paper is as follows.
In \S\ref{sec:grassmann} we provide context for the results above,
compare Conjecture \ref{conj:Wr} above
with the Grassmann convexity conjecture
and obtain corollaries of our main results
which, we hope, will further motivate their interest.
We start \S\ref{sec:proofs} by reviewing and motivating
the constructions above.
We also recall the concept of {\em total positivity},
which will be used in the proof of both main theorems.
There is of course a vast literature concerning
the subject of total positivity (among many others, \cite{bfz}).
We define the set $\fW$ of {\em admissible cyclic words}
and a correspondence from an open dense subset of $\Lo_n^1$ to $\fW$.
We define {\em admissible moves} in $\fW$ (Definition \ref{def:admissible}).
The definition is combinatorial (and simple) and
when following a flag-convex curve $\Gamma$ we perform admissible moves
(Lemma \ref{lemma:admissible}).
We then introduce our main technical tool,
the {\em rank function} (Definition \ref{def:rank}).
The rank function is integer valued and defined for admissible cyclic words;
the definition is combinatorial and elementary, but long.
Several basic properties are given,
which admit simple but sometimes long proofs.
For instance, the rank of the totally positive word is $0$
and the rank of the totally negative word is $2(n-2)$;
other words give intermediate values
(Lemma \ref{lm:rankstep}).
The crucial property is that,
when following a flag-convex curve $\Gamma$,
the rank of $\Gamma(t)$ is always non-increasing
and is strictly decreasing at roots of $m_2$
(Proposition \ref{prop:rank_move}).
The proof of Theorem~\ref{th:main} is now easy.

In \S\ref{sec:La} we prove Theorem \ref{th:La}.
We first recall several concepts and results from \cite{GoSa0}.
In particular, we define the {\em multiplicity vector}
$\mult(\sigma)$ of a permutation $\sigma \in S_n$.
The definition of $\mult(\sigma)$ is combinatorial; however 
 Theorem 4 from \cite{GoSa0} provides
an algebraic or geometric interpretation.
Indeed, if $\Gamma$ is flag-convex and $\Gamma(t_0) \in \Bru_\sigma$
then $\mult_k(\sigma)$ is the multiplicity of the zero $t = t_0$
of the function $m_k$.
We use the notation $\Bru_\sigma$ for the Bruhat cell
corresponding to the permutation $\sigma$
(see again \cite{GoSa0}).
We first state a warm-up special case, Proposition \ref{prop:L}.
We then prove additional statements
(Lemmas \ref{lemma:goodstep} and \ref{lemma:goodstepLa})
concerning pairs of permutations
$\sigma_0 \triangleleft \sigma_1 = \sigma_0 a_i$:
here $\triangleleft$ denotes the Bruhat order and
$a_i$ is a transposition, a generator of $S_n$.
The change in multiplicities between $\sigma_0$ and $\sigma_1$
is described by Lemma 2.4, also from \cite{GoSa0}.
Only the two results above are needed from \cite{GoSa0};
duplicating the proof of these results here would lead us too far astray.
A simple induction then settles
Proposition~\ref{prop:L} and Theorem~\ref{th:La}
(the induction is described in
Lemmas \ref{lemma:goodmatrix} and \ref{lemma:goodmatrixLa}).
Besides the references we already mentioned other relevant results
can be found in e.g. \cite{SeSh} and \cite{BaMaPo}.

We finish the introduction with the following tantalizing question: 

\smallskip
\noindent
{\it Is it possible to extend the present approach (see especially  \S\ref{sec:proofs}) 
from the case of $G_{2,n}$ to other Grassmannians?} 
\medskip

\noindent
Acknowledgements.
The authors thank Victor Goulart for his help with the revision of the text.
The first author wants to acknowledge support of CNPq,
CAPES and Faperj (Brazil) and
to express his sincere gratitude to
the Department of Mathematics, Stockholm University
for the hospitality in November 2018.  
The second author wants to acknowledge
the financial support of his research provided by
the Swedish Research Council grant 2016-04416.
The third author is supported by the NSF grant DMS-1702115.

\section{The Grassmann convexity conjecture}
\label{sec:grassmann}

Conjecture~\ref{conj:Wr}  has an equivalent reformulation
called the {\it Grassmann convexity conjecture}
first suggested in \cite {ShSh}, Main Conjecture 1.1.
To state it, we need  some further  definitions. 

\begin{definition} {A smooth closed curve $\ga: \Ss^1\to \bP^{n-1}$ is 
called {\it
locally convex} if, for any hyperplane $H\subset 
\bP^{n-1}$,  the local
 multiplicity of the intersection of $\ga$ with $H$ 
at any of the intersection points $p\in \ga \cap H$ 
does not exceed
$n-1=\dim \bP^{n-1}$ and {\it globally convex} if the
above condition holds for the sum of all local multiplicities, see e.g. \cite{ShSh}.}
\end{definition}

 Below we will often refer to globally convex curves simply as {\it convex}. The above notions are directly generalized to smooth non-closed curves, i.e. $\ga: I \to \bP^{n-1}$.
 
\begin{remark}
{Local convexity of $\ga$ is an easy requirement equivalent to the
non-degeneracy of the
osculating Frenet $(n-1)$-frame of $\ga$, i.e. to the linear independence 
of
$\ga^\prime(t),\dots ,\ga^{(n-1)}(t)$ at all points $t\in \Ss^1$. Global convexity 
is a
rather nontrivial
property studied under different names 
since the beginning of the last century. 
(There exists a vast literature on convexity and the classical 
achievements
are well
summarized in \cite {Co}. For more recent developments 
see e.g.
\cite {Ar}).}
\end{remark} 

Denote by $G_{{k,n}}$  the usual Grassmannian of real $k$-dimensional 
linear subspaces in $\bR^n$ (or equivalently, of real $(k-1)$-dimensional projective subspaces in
$\bP^{n-1}$).

\begin{definition}
{Given an $(n-k)$-dimensional linear subspace $L\subset \bR^n$, we define 
{\it the Grassmann hyperplane $H_{L}\subset G_{k,n}$ associated to 
$L$}  as the
set of all
$k$-dimensional linear subspaces in $\bR^n$ non-transversal to $L$.}
\end{definition}

\begin{remark}{\rm 
  The concept of Grassmann hyperplanes is well-known  in Schubert calculus, (see e.g.
\cite {Ca} and  \cite{ShSh}.) 
More exactly,  $H_L\subset G_{k,n}$ coincides with the union of all Schubert cells of
positive codimension
 constructed using any complete flag containing $L$ as 
a linear 
subspace. The complement $G_{k,n}\setminus  H_L$ is the open Schubert cell isomorphic to the standard affine 
chart
in $G_{k,n}$. 
  By duality, $H_L\subset G_{k,n}$ is isomorphic to $H_{L^\prime}\subset
G_{n-k,n}$ where $L^\prime$ is a $k$-dimensional linear subspace in $\bR^n$. }
\end{remark}

\begin{remark} {\rm A usual hyperplane $H\subset \bP^{n-1}$ is a particular
case of a  Grassmann hyperplane if we interpret $H$ as the set of all 
points
non-transversal (i.e. belonging) to $H$. $H$ itself can be   considered as a point in
$(\bP^{n-1})^\star$.} 
\end{remark}

\begin{definition} {A smooth closed curve $\Ga: \Ss^1\to G_{k,n}$ is called 
{\it locally
Grassmann-convex} if the local multiplicity of the intersection of $\Ga$ 
with
any Grassmann
hyperplane $H_L\subset G_{k,n}$ at any of its intersection points does not exceed $k(n-k)=\dim G_{k,n}$, and
{\it globally Grassmann-convex} 
if  the above condition holds
for the sum of all local  multiplicities, see \cite{ShSh}.}
\end{definition}

Below we refer to globally Grassmann-convex curves as {\it Grassmann-convex}.
The notions are directly generalized to smooth non-closed curves,
i.e. $\Ga: I \to G_{k,n}$. 

\begin{definition}
{Given a locally convex curve $\ga:\Ss^1\to \bP^{n-1}$
and a positive integer $1\le k\le n-1$,  
we 
define its {\it  $k$th osculating Grassmann curve}
 $osc_k\ga: \Ss^1\to G_{k,n}$ as the curve formed by the 
$(k-1)$-dimensional projective 
subspaces osculating  the initial
$\ga$. }
\end{definition}

For any $k=1,\dots,n-1$, the curve $osc_k\ga$ is well-defined   
due to the local convexity of the curve $\ga$. 

\begin{CONJ}[Grassmann convexity conjecture]\label{conj:Gr} 
For any  convex  curve $\ga:\Ss^1\to
\bP^{n-1}$ (resp. $\ga: I \to
\bP^{n-1}$) and any $1\le k \le n-1$,  its  osculating curve
$osc_k\ga: \Ss^1\to G_{k,n}$  (resp. $osc_k\ga: I\to G_{k,n}$) is Grassmann-convex. 
\end{CONJ} 

The equivalence of Conjectures~\ref{conj:Wr} and ~\ref{conj:Gr} is straightforward and, in particular,   is explained in \cite{ShSh}.

To explain this equivalence we need the following. A curve $\gamma=(\gamma_1,\dots,\gamma_n):I\to \bR^n$ is called 
\emph{non-degenerate} if at every point $t\in I$ 
its osculating frame
$\{\gamma(t),\gamma'(t),\gamma''(t),\dots,\gamma^{n-1}(t)\}$
is non-degenerate.
This is equivalent to the fact that   its Wronski matrix $W(t)=W(\gamma_1(t),\dots,\gamma_n(t))$ has full rank. 
 
 Non-degenerate curves can be trivially identified with fundamental solutions of linear differential equations \eqref{eq:disc}. In particular,  we call a non-degenerate $\gamma$ \emph{disconjugate} if the corresponding equation \eqref{eq:disc} is disconjugate. On the other hand, it is obvious that  $\gamma$ is non-degenerate/disconjugate if and only if its projectivization  is locally convex/convex.
 
 Moreover, given a non-degenerate curve $\gamma=(\gamma_1,\dots,\gamma_n):I\to \bR^n$ and an integer $1\le k<n$, the zeros of the Wronskian $W(\ga_1,\dots, \ga_k)$ can be interpreted as the moments when the $k$th osculating Grassmann curve $osc_k\ga: I \to G_{k,n}$ intersects an appropriate Grassmann hyperplane;  for more details on $k=2$, see Section~\ref{sec:proofs}.   
  Observe that Conjecture~\ref{conj:Gr} is trivially satisfied for $k=1$ and $k=n-1$. 

\medskip

  Notice additionally that Theorem~\ref{th:main} admits the following natural interpretation, compare loc. cit. 

\begin{definition}  {Given a generic curve $\ga:\Ss^1\to \bP^{n-1}$,  we define its 
{\it
standard discriminant} 
$D_\ga\subset \bP^{n-1}$ to be the hypersurface consisting of all 
 subspaces  of codimension 2 osculating  $\ga$. (Here `generic' means having a
non-degenerate osculating $(n-2)$-frame at every point.) }
\end {definition}

\begin{definition}
{By  the \emph {$\bR$-degree}  of a real closed
(algebraic or non-algebraic) hypersurface $\cH\subset \bR^n$ (resp. $\cH \subset \bR P^{n-1}$) without boundary we mean  the supremum of the cardinality of $\cH \cap L$  taken over all lines $L \subset \bR^n$ (resp. $L \subset \bR P^{n-1}$)  such that $L$ intersects $\cH$ transversally. (Observe that the $\bR$-degree of a hypersurface can be infinite. Discussions of this notion can be found in \cite{LaShSh}).} 
\end{definition}

\begin{corollary}\label{cor:degree} For any closed convex curve $\ga:\Ss^1\to \bP^{n-1}$, the $\bR$-degree of its discriminant $D_\ga$ equals $2n-4$. 
\end{corollary}

\noindent
{\it Basic notions of the multiplicative Sturm separation theory.}
\smallskip
 Following  \cite {ShSh}, let us now  recall the set-up of this theory,  an early version of which can be found in \cite {Sh}.   
 
 \smallskip
 Denote by $\Fl_n$ the space of complete real flags in $\bR^n$. We say that two complete flags $f_1, f_2 \in \Fl_n$   are {\it transversal}  if, for any $1\le i\le n-1$, the intersection of the $i$-dimensional subspace of $f_1$ with the $(n-i)$-dimensional subspace of $f_2$ coincides with the origin. Otherwise the flags $f_1, f_2 \in \Fl_n$   are called {\it non-transversal}. 

\begin{definition}
{ 
Given a locally convex curve $\ga:\Ss^1\to \bP^{n-1}$,   
define its {\it  osculating flag curve}
 $\gamma_{\Ff} : \Ss^1\to \Fl_n$ to be the curve formed by the 
complete flags osculating  $\ga$, see e.g. \cite{Sh}. (The curve $\gamma_{\Ff}$ is well-defined   
due to the local convexity of $\ga$; similar notion obviously exists for non-closed locally convex curves).}
\end{definition}

 For a non-degenerate curve $\Gamma :I\to\RR^n$ (or, equivalently, for its projectivization $\gamma: I \to \bP^{n-1}$) and any fixed flag $f\in \Fl_n$,  denote by $\sharp_{\ga, f}$   the number of {\it moments of non-transversality}   between $\gamma_{\Ff}$ and $f$, where  $t\in I$ is called a moment of non-transversality if the complete flags $\gamma_{\Ff}(t)$ and $f$ are non-transversal.  Define $\sharp_\ga=\sup_{f\in \Fl_n} \sharp_{\ga, f}$. 

\medskip
           The following two lemmas provide criteria for (non-)disconjugacy of linear ordinary differential equation or, equivalently, (non-)convexity of projective curves, compare \cite {Lev}.

\begin{LEM}[see \cite {Sh}]
\label{lem:disconj}
A locally convex curve $\ga: \Ss^1\to \bP^{n-1}$  (resp. $\ga: I\to \bP^{n-1}$)
is {\it globally convex} if and only if,
for all $t_1\neq t_2 \in \Ss^1$ (resp. $t_1\neq t_2 \in I$),
the flags $\gamma_{\Ff}(t_1)$ and $\gamma_{\Ff}(t_2)$ are transversal. 
\end{LEM}

\begin{LEM}[see \cite {Sh}]
\label{lem:conj}
A locally convex curve $\ga:\Ss^1\to \bP^{n-1}$ (resp. $\ga: I\to \bP^{n-1}$)
is {\it not globally convex}  if and only if,
for any complete flag $f\in \Fl_n$,
there exists $t\in \Ss^1$ (resp. $t\in I$) such that
$f$ and $\gamma_{\Ff}(t)$ are non-transversal.
\end{LEM}

The next claim appears to be new,
but closely related to Conjecture \ref{conj:Wr}.

\begin{conjecture}
\label{conj:sturm}
For any convex curve $\ga: \Ss^1\to \bP^{n-1}$ (resp. $\ga: I\to \bP^{n-1}$),
one has
$$\sharp_\ga=\frac{n^3-n}{6}.$$
\end{conjecture}

\smallskip
Conjecture~\ref{conj:sturm} is obvious for $n=2$ and easy for $n=3$;
we shall soon see some further support for it. 
The following result follows directly from Theorem \ref{th:La}.

\begin{corollary}\label{coro:ineq}
For any convex curve $\ga: \Ss^1\to \bP^{n-1}$ or
$\ga: I\to \bP^{n-1}$ we have
$$\sharp_\ga\ge \frac{n^3-n}{6}.$$
\end{corollary}

\begin{proof}
We may assume $I$ open and $\ga$ associated to $\Gamma: I \to \Lo_n^1$,
$\Gamma$ flag-convex
(this correspondence is discussed in the first paragraphs
of Section \ref{sec:proofs}).
Apply Theorem \ref{th:La} to obtain a matrix $L_1$
such that, for $L_1\Gamma$,
the functions $m_k$ have, among them, $\frac{n^3 - n}{6}$ roots.
The matrix $L_1$ defines a flag $f_1$:
the moments of non-transversality beween $\gamma$ and $f_1$
are precisely the roots of the functions $m_k$
for $L_1\Gamma$, as above. 
\end{proof}

\medskip
Combining Theorems~\ref{th:main} and Corollary~\ref{coro:ineq}
we get  the following.

\begin{corollary}
\label{th:4-5}
Conjecture~\ref{conj:sturm} holds for $n=4$ and $n=5$.
\end{corollary}

\begin{proof}
We use the notation of Theorem \ref{th:La},
particularly the functions $m_k$.
We know that $m_1$ and $m_{n-1}$ have at most $n-1$ zeroes each.
From Theorem \ref{th:main}, we know that $m_2$ and $m_{n-2}$
have at most $2(n-2)$ zeroes each.
For $n \le 5$, this covers all possible values of $k$.
This implies that $\sharp_\ga \le \frac{n^3-n}{6}$
for any convex curve $\ga: \Ss^1\to \bP^{n-1}$.
Corollary \ref{coro:ineq} gives us the other inequality.
\end{proof}




To finish this section, let us mention that it is well-known that,
for general equations~\eqref{eq:disc} of order exceeding $2$,
the number and location of the zeros of their different individual solutions
can be quite arbitrary.  
On the other hand,  for any equation~\eqref{eq:disc},
one  can split its time interval $I$ into maximal disjoint subintervals
on each of which \eqref{eq:disc} is disconjugate. 
In order to get a meaningful comparison theory, 
instead of looking at the individual solutions
we could compare different fundamental solutions of ~\eqref{eq:disc},
i.e. count the number of moments of non-transverality of the flag curve of ~\eqref{eq:disc} with  different complete flags.
This approach leads to the following claim which is a  new type of generalization  of the classical Sturm separation theorem from the case of  linear ode of order $2$ to the case of arbitrary order, comp. \cite{Sh}.

            \begin{CONJ}[see \cite{ShSh}]\label{conj:SturmN}  For $n\ge 2$, let $\ga: \Ss^1\to \bP^{n-1}$ (resp. $\ga: I\to \bP^{n-1}$) be a locally, but not globally convex curve.
            Then, for any pair of complete flags $f_1$ and $f_2$, 
            \begin{equation}\label{eq:Sturm}
                \sharp_{\ga, f_1}\leq \dfrac{n^3-n+6}{6}\cdot \sharp_{\ga,  f_2}.
            \end{equation}
\end{CONJ}

Observe that (if settled)  Conjecture~\ref{conj:sturm} combined with Lemma~\ref{lem:conj} will imply Conjecture~\ref{conj:SturmN}.

\medskip
\noindent

\section{Proof of Theorem \ref{th:main}} \label{sec:proofs}



In this section we follow the notation of \cite{GoSaX, GoSa0}
and use  matrix realizations of flag curves.
(Such realizations were already used in the earlier papers by the authors). 

Observe that we can assume that for any convex curve $\ga: \Ss^1\to \bP^{n-1}$ (or  $\ga: I\to \bP^{n-1}$), its osculating flag curve $\ga_{\Ff}: \Ss^1\to \Fl_n$ (resp. $\ga_{\Ff}: I\to \Fl_n$) lies completely in some top-dimensional  Schubert cell in $\Fl_n$.
To see that, depending on whether one considers the case of $\Ss^1$ or $I$, let us either fix an arbitrary point $\tau\in \Ss^1$ or the left endpoint $\tau \in I$. 
Take the flag  $\ga_{\Ff}(\tau)\in \Fl_n$ as the complete flag defining the top-dimensional Schubert cell in $\Fl_n$. (In other words, we take all complete flags transversal to $\ga_{\Ff}(\tau)$.) 
By Lemma~\ref{lem:disconj}, for any $\nu\in \Ss^1$ (resp. $\nu\in I$) different from $\tau$, the flags $\ga_{\Ff}(\tau)$ and $\ga_{\Ff}(\nu)$ are transversal, which means that the latter flag lies in the top-dimensional Schubert cell in $\Fl_n$ with respect to the former flag. Thus the whole flag curve $\ga_{\Ff}$, except for one point $\ga_{\Ff}(\tau)$, lies in this top-dimensional cell.

Top-dimensional cells in $\Fl_n$ are standardly identified with $\Lo_n^1$, where  $\Lo_n^1$ is the nilpotent Lie group of real lower triangular
$n\times n$ matrices with all diagonal entries equal to $1$.
This group can be interpreted as the tangent space to $\Fl_n$
at any fixed chosen flag $f$. 
Alternatively, the usual $LU$ and $QR$ decompositions
define diffeomorphisms $\bQ: \Lo_n^1 \to \cU_1$ and
$\bL: \cU_1 \to \Lo_n^1$, where $\cU_1 \subset \Fl_n$
is a top dimensional cell (see \cite{GoSa0}).

Recall that the $LU$ decomposition of an invertible matrix $A \in \GL_n$
is a pair of matrices $L$ and $U$ such that $A = LU$,
$L \in \Lo_n^1$ is lower triangular with diagonal entries equal to $1$
and $U \in \Up_n$ is upper triangular (and invertible).
There exists a neighborhood
of the identity matrix where such a decomposition
is smoothly and uniquely defined.
The set $\cU_1 \subset \SO_{n}$ is the intersection
of this open neighborhood with the subgroup $\SO_{n}$.
We abuse notation by writing either $\cU_1 \subset \Fl_n$,
$\cU_1 \subset \SO_n$ or $\cU_1 \subset \Spin_n$
since the manifolds on the right hand side are locally identified:
$\Spin_n$ is a double cover of $\SO_n$
and $\SO_n$ is a finite cover of $\Fl_n$.
Also, a $QR$ decomposition  of an invertible matrix $A \in \GL_n$
is a pair of matrices $Q$ and $R$ such that $A = QR$,
$Q \in O_n$ is orthogonal and
$R \in \Up_n^{+}$ is upper triangular with positive diagonal entries:
this decomposition is smoothly and uniquely defined in $\GL_n$.
These decompositions allow us to define the diffeomorphisms
$\bQ: \Lo_n^1 \to \cU_1$ and $\bL: \cU_1 \to \Lo_n^1$.


\smallskip
The following statement can be found in e.g. \cite{Sh, GoSaX, GoSa0}.
Recall that $\cJ$ is the set of $n\times n$ real matrices $A$
such that the entry $A_{ij}$
is positive if $i = j+1$ and $0$ otherwise.
In other words, $A \in \cJ$ if and only if
$A$ has strictly positive subdiagonal entries
(i.e. entries in positions $(j+1,j)$) and zero entries elsewhere.

\begin{LEM}\label{lem:char}
Consider an interval $I \subseteq \RR$ and
a smooth curve $\Gamma: I \to \Lo_n^1$. 
Then $\Gamma$ is the osculating flag curve  of 
a convex projective curve $\gamma: I \to \bP^{n-1}$
(i.e.,  $\gamma_{\Ff} = \bQ \circ \Gamma$ where $\bQ: \Lo_n^1 \to \cU_1 \subset \Fl_n$ is  the above diffeomorphism)  
if and only if, for every $t \in I$,
the logarithmic derivative
$(\Gamma(t))^{-1}\Gamma'(t)$ belongs to the set $\cJ$.
\end{LEM}


Let us call the osculating flag curves obtained 
by taking the  flags osculating  convex projective curves
{\it flag-convex} (or sometimes just {\it convex}). 
In other words,
a curve $\Gamma: I \to \SO_n$ is {\it convex}
if and only if
there exists a convex curve $\gamma: I \to \SO_n$
with $\Gamma = \gamma_{\Ff}$.
It turns out that $\Gamma: I \to \SO_n$ is convex
if and only if there exists $Q_0 \in \SO_n$
such that $Q_0 \Gamma(t) \in \cU_1$
for all $t$ in the interior of $I$. 
(Compare with Lemma \ref{lem:disconj};
see also \cite{GoSa0}).
We sometimes abuse notation by identifying 
$\Lo_n^1$ with $\cU_1$ (through the diffeomorphism $\bQ$)
and therefore $\Gamma$ with $\bQ \circ \Gamma$.

\medskip
Given a flag-convex curve $\Gamma: I \to \Lo_n^1$,
define the function $m_{\Gamma,k} = m_k: I \to \RR$, given by 
$$m_k(t) = \det(\swminor(\Gamma(t),k)).$$
Recall that $\swminor(L,k)$ is the $k\times k$ submatrix of $L$
formed by its last $k$ rows and its first $k$ columns. 

Observe now that if we interpret $ \Lo_n^1$
as the top-dimensional cell in $\Fl_n$
with respect to some fixed flag $g\in \Fl_n$,
then the moments of non-transversality of
the flag curve $\Gamma$ with the $(n-k)$-dimensional linear subspace 
 belonging to  $g$ are exactly the zeros of $m_k(t)$.

Thus, Conjecture \ref{conj:Wr} is equivalent to saying that
for any $n$, for any $k \le n$, and
for any flag-convex curve $\Gamma: I \to \Lo_n^1$,
the number of real zeroes of the function $m_k(t)$ where $t\in I$,
is at most $k(n-k)$.

\begin{remark}
\label{remark:star}
Consider a flag-convex curve $\Gamma: I \to \Lo_n^1$.
Set $-I = \{t; -t \in I\}$;
let $P_\eta$ be the permutation matrix corresponding
to the top permutation $\eta$
so that $(P_\eta)_{i,n+1-i} = 1$ ($1 \le i \le n$),
with the other entries equal to zero.
Define
\[ \Gamma_{\star}: -I \to \Lo_n^1, \quad
\Gamma_{\star}(t) = P_\eta (\Gamma(-t))^{-\top} P_\eta; \]
a straightforward computation verifies that
the curve $\Gamma_{\star}$ is also flag-convex.
We also have
\begin{align*}
m_{\Gamma_{\star},k}(t) &=
\pm \det(\Gamma_{\star}(t)e_1, \ldots, \Gamma_{\star}(t)e_k,
e_1, \ldots, e_{n-k})
= \\
&=
\pm \det(e_1, \ldots, e_k,
(\Gamma_{\star}(t))^{-1} e_1, \ldots, (\Gamma_{\star}(t))^{-1}e_{n-k})
= 
\pm m_{\Gamma,n-k}(-t). 
\end{align*}
Thus, Conjecture \ref{conj:Wr} for $k = k_0$ implies
the same conjecture for $k = n-k_0$.
\end{remark}

%
%


%
 
Define the open and dense 
subset $\Lo_n^o\subset \Lo_n^1$ given by
$$\Lo_n^o=\{X\in \Lo_n^1 \;|\;
\forall k\in [1,n-1], m_k(X)\ne 0 \}.$$

\begin{remark}
In the notation of \cite{GoSa0}, we have
$\Lo_n^o = \bigsqcup_{q \in \Quat_n} \bQ^{-1}[\Bru_{q\acute\eta}]$.
The set $\Lo_n^o$ is a disjoint union of 
 finitely many connected components.
These connected components were counted in \cite{ShShV4}
and several follow-up papers.
In particular, their number equals $2, 6, 20, 52$ for $n=2,3,4,5$ resp.
and it is equal to $3\times 2^{n-1}$ for all $n\ge 6$. 
\end{remark}
 
We will  specially distinguish two of these connected components. 
Recall that a matrix $L_0 \in \Lo_n^1$ is \emph{totally positive}
provided that,
if a minor is nonzero for some $L \in \Lo_n^1$,
then the corresponding minor is strictly positive for $L_0$ (see \cite{ShSh2});
the set $\Pos \subset \Lo_n^1$ of totally positive matrices
is a contractible connected component of $\Lo_n^o$.
Similarly, the set $\Neg \subset \Lo_n^1$ of totally negative matrices
is another contractible  connected component of $\Lo_n^o$.
For $L \in \Lo_n^1$, we have that $L \in \Neg$ if and only if
$PLP \in \Pos$, where the diagonal matrix $P$ is given by   
$P=\diag(1,-1,1,\cdots, (-1)^{n-1})$. 
Equivalently, for $L \in \Lo_n^1$,
$L \in \Neg$ if and only if $L^{-1} \in \Pos$.
In Lemma \ref{lem:ShGoSa0} we provide an alternative characterization
of the subsets $\Pos, \Neg \subset \Lo_n^1$;
here $\operatorname{id} \in \Lo_n^1$ stands for the identity matrix.


 \begin{LEM}[see \cite {Sh}, \cite{GoSa0}]
\label{lem:ShGoSa0}
If $\Gamma: I \to \Lo_n^1$ is flag-convex and $\Gamma(0) = \operatorname{id}$
then $\Gamma(t) \in \operatorname{Pos}$ for $t > 0$
and $\Gamma(t) \in \operatorname{Neg}$ for $t < 0$.
Conversely, if $L_1 \in \operatorname{Pos} \subset \Lo_n^1$
and $L_{-1} \in \operatorname{Neg} \subset \Lo_n^1$ then
there exists a smooth flag-convex curve  $\Gamma: \RR \to \Lo_n^1$ 
such that $\Gamma(-1) = L_{-1}$, $\Gamma(0) = \operatorname{id}$
and $\Gamma(1) = L_1$.
 \end{LEM}

 
Recall that, for any locally convex curve $\ga :I\to \bR P^{n-1}$,
we denote by $\ga_\Ff(t)$ its osculating flag curve and by
$osc_2\ga(t):I\to G_{2,n}$ 
the osculating Grassmann curve obtained by taking the span
of the  first two  columns of $\gamma_\Ff(t)$. 
Fix the  subspace
$L=\operatorname{span}\langle e_1,\dots,e_{n-2}\rangle \subset \bR^n$
of codimension $2$.
Observe that the intersection of $osc_2\ga $ with $H_L$ is  given by the equation $m_2(t)=0.$ Here $H_L\subset G_{2,n}$ is the Grassmann hyperplane associate with the latter $(n-2)$-dimensional $L$. 

\smallskip
In what follows, 
instead of considering the curve $osc_2\ga(t):I\to G_{2,n}$
we present a related construction.

\smallskip 
Let $\C = \Lo_{(2\times n)}^1 \subset \RR^{2 \times n}$ be the space
of real $(2 \times n)$ matrices $X$
such that $X_{1,n-1} = X_{2,n} = 1$ and $X_{1,n} = 0$.
There is a natural projection $\Pi: \Lo_n^1 \to \C$ taking $L$
to the submatrix formed by its last two rows:
$(\Pi(L))_{i,j} = L_{i+n-2,j}$.
Alternatively, $\Pi(L) = X_0 L$ where $X_0 = \Pi(\operatorname{id}) \in \C$
is the matrix whose only nonzero entries are
$(X_0)_{1,n-1} = (X_0)_{2,n} = 1$.
Equivalently,
let $H_0, H_1 \subset \Lo_n^1$ be the subgroups defined by
\begin{align*}
H_0 &= \{ L \in \Lo_n^1 \;|\; \forall(i,j),
1 \le j < i \le n, i > n-2 \to L_{i,j} = 0 \}; \\
H_1 &= \{ L \in \Lo_n^1 \;|\; \forall(i,j),
1 \le j < i \le n-2 \to L_{i,j} = 0 \}.
\end{align*}
If $L_0 \in H_0$, $L_1 \in H_1$ and $L = L_0 L_1$
then
$L_{i,j} = (L_0)_{i,j}$ if $i \le n-2$ and
$L_{i,j} = (L_1)_{i,j}$ if $i > n-2$.
Thus, any $L \in \Lo_n^1$ can be uniquely written as a product
$L = L_0L_1$ with $L_0 \in H_0$ and $L_1 \in H_1$.
The restriction $\Pi|_{H_1}: H_1 \to \C$ is thus a bijection.
The space $\C$ is naturally identified with $H_0 \backslash \Lo_n^1$,
the set of right cosets of the form $H_0 L$, $L \in \Lo_n^1$;
the map $\Pi$ is now the natural quotient map
$\Lo_n^1 \to H_0 \backslash \Lo_n^1$.

Below  we will treat $X \in \C$ as an $n$-tuple of real column vectors:
$X = (v_1, \ldots, v_n)$, $v_i \in \RR^2$.
In other words, $\C \subset (\RR^2)^n$ is the set of $n$-tuples
$X = (v_1, \ldots, v_n)$ satisfying
\[ v_{n-1} = \begin{pmatrix} 1\\a\end{pmatrix}, \qquad
v_n = \begin{pmatrix} 0\\1\end{pmatrix} \]
for some $a \in \RR$.
Clearly, $m_2(L) = m_2(\Pi(L))$ for all $L \in \Lo_n^1$;
the map $m_2$ is thus well-defined as $m_2: \C \to \RR$. 
For any set $Y = \{i < j\} \subset \{1, 2, \ldots, n\}$ (with $|Y| = 2$),
we define a function $m_Y: \C \to \RR$:
for $X = (v_1, \ldots , v_n)$, set $m_Y(X) = \det(v_i,v_j)$.
Notice that $m_2 = m_{\{1,2\}}$;
for all such $Y = \{i,j\}$, set $\sum Y = i+j$.


A smooth curve  $\Gamma_2: I \to  \C$ is called \emph{flag-convex}
if there exists a flag-convex curve $\Gamma: I \to \Lo_n^1$
such that $\Gamma_2 = \Pi \circ \Gamma$.
The set $\cL$ of flag-convex curves $\Gamma_2: I \to \C$
has a smooth Banach manifold structure,
inherited through $\Pi$ from 
the set of flag-convex curves $\Gamma: I \to \Lo_n^1$.
We are interested in proving that if $\Gamma_2$ is flag-convex
then $m_2 \circ \Gamma_2$ has at most $2(n-2)$ real zeroes.

Define $\Pos_2 = \Pi[\Pos] \subset \C$ and $\Neg_2 = \Pi[\Neg] \subset \C$.
Similarly, we say that $X \in \C$ is \emph{totally positive} 
(resp. \emph{negative}) if $X \in \Pos_2$ (resp. $X \in \Neg_2$). 
The following observation is straightforward.

 

 \begin{lemma} 
A $2\times n$-matrix $X \in \C$
lies in $\operatorname{Pos_2}$ if and only if 
 \begin{itemize} 
\item  $X_{ij}>0$ for all $i\in[1,2]$, $j\in[1,n-2]$ and $i=2$, $j=n-1$;
\item  $m_Y(X) > 0$ for all $Y \subset \{1, \ldots, n\}$, $|Y| = 2$.
 \end{itemize}

 
A $2\times n$-matrix $X \in \C$
lies in  $\operatorname{Neg}_2$  if and only if 
 \begin{itemize}
\item 
$(-1)^{(i+j)} X_{ij}>0$ for all $i\in[1,2]$, $j\in[1,n-2]$ and $i=2$, $j=n-1$;
\item 
$(-1)^{(1+\sum Y)} m_Y(X) > 0$ for all $Y \subset \{1, \ldots, n\}$, $|Y| = 2$.
\end{itemize}
 \end{lemma}

Interpret $\C$ as a set of $n$-tuples of vectors $v_i \in \RR^2$.
Let $\C_1 \subset \C$ be the open dense subset of such $n$-tuples
such that, for all $i$, $v_i \ne 0$.
Notice that the complement $(\C \smallsetminus \C_1) \subset \C$
is a union of finitely many submanifolds of codimension $2$.
Consider the set $\cL$ of flag-convex curves $\Gamma_2: I \to \C$
and its subset $\cL_1 \subset \cL$ of curves 
with image contained in $\C_1$.
Then $\cL_1 \subset \cL$ is also open and dense;
moreover, the codimension of the complement is $1$.
In particular, an estimate on the number of zeroes of $m_2 \circ \Gamma_2$
for $\Gamma_2 \in \cL_1$
automatically implies the same estimate for $\Gamma_2 \in \cL$.

Let $\C_2 \subset \C_1$ be the open dense subset of $n$-tuples
$X = (v_1, \ldots, v_n)$ such that
$m_Y(X) = 0$ holds for at most one such set $Y$;
let $\C_3 \subset \C_2$ be the open dense subset of $n$-tuples
$X = (v_1, \ldots, v_n)$ such that,
for all such sets $Y$, $m_Y(X) \ne 0$.
In other words, for $X = (v_1, \ldots, v_n) \in \C_1$, we have
$X \in \C_3$ if and only if the vectors $v_1, \ldots, v_n$
are pairwise linearly independent.

Notice that the complement of $\C_2$ is a union of finitely many
submanifolds of codimension at least $2$.
Thus $\C_2$ is path-connected and
generic flag-convex curves $\Gamma_2: I \to \C$
are of the form $\Gamma_2: I \to \C_2$.
As we shall see, $\C_3$ has exactly $2^{n-2}\cdot (n-1)!$
connected components, all contractible.
Connected components of $\C_3$ are labeled
by signed cyclic words, as we explain below.


Namely, consider cyclic words $w$ of length $2n$
in the alphabet $1,\dots,n,1',\dots,n'$
such that $w$ contains each letter exactly once.
We say that such a word $w$ is  {\it admissible} (or {\it odd})
if, for every $i$, there are exactly $n-1$ other letters
between the letters $i$ and $i'$:
let $\fW$ denote the set of admissible words.
For example, for $n = 3$, we have
\begin{align*}
\fW &= \{
123'1'2'3,12'3'1'23,1'23'12'3,1'2'3'123, \\
&\qquad 213'2'1'3,2'13'21'3,21'3'2'13,2'1'3'213 \}.
\end{align*}
In general, we have $|\fW| = 2^{n-1}\cdot (n-1)!$
(fix $n$ and $n'$;
choose one among the $(n-1)!$ permutations of $\{1,\ldots,n-1\}$
to fill in the gap between $n$ and $n'$;
for $i$ from $1$ to $n-1$ choose the positions of $i$ and $i'$).

Words in $\fW$ should be imagined as written along a circle,
always counter-clockwise.
Given $i < j$, we say that we walk from $i$ to $j$ counter-clockwise in $w$
in there are fewer than $n-1$ letters after $i$ and before $j$:
in this case we write $m_{\{i,j\}}(w) > 0$
(otherwise $m_{\{i,j\}}(w) < 0$).
Equivalently, $m_{\{i,j\}}(w) > 0$ if and only if 
one encounters the triple $i,\dots, j ,\dots , i'$
when reading $w$ (counter-clockwise);
of course, $m_{\{i,j\}}(w) < 0$ if and only if 
one encounters instead the triple $i,\dots, j' ,\dots , i'$.
Thus, for instance, if $n = 5$ and $w = 13'4'25'1'342'5$
then
$m_{\{1,2\}}(w) > 0$, $m_{\{1,3\}}(w) < 0$,
$m_{\{1,4\}}(w) < 0$, $m_{\{1,5\}}(w) < 0$,
$m_{\{2,3\}}(w) > 0$, $m_{\{2,4\}}(w) > 0$,
$m_{\{2,5\}}(w) < 0$,
$m_{\{3,4\}}(w) > 0$, $m_{\{3,5\}}(w) > 0$ and
$m_{\{4,5\}}(w) > 0$.
If $w$ is a word, we assume that $m_Y(w) \in \{\pm 1\}$.
Let $\fW^{+} \subset \fW$ be the set of admissible words $w$
for which $m_{\{n-1,n\}}(w) > 0$;
we have $|\fW^{+}| = 2^{n-2}\cdot (n-1)!$.

We now show how to assign a word $w(X) \in \fW^{+}$
to each $X \in \C_3$.
Given $X = (v_1, \ldots, v_n)$,
set $\nu(X) = (\hat v_1, \ldots, \hat v_n) \in  (\Ss^1)^n$
where $\hat v_i = v_i/|v_i| \in \Ss^1$.
Let $\Omega_n = \nu[\C_3]$;
in other words,
$\Omega_n\subset (\Ss^1)^n$ is the set of configurations
of $n$ pairwise linearly independent labeled points on $\Ss^1$
such that point $n$ is $(0,1)$
and point $n-1$ has coordinates $(x,y)$ with $x>0$.
Given $X = (v_1, \ldots, v_n)$,
label the point $\hat v_i$ by $i$ and the point $-\hat v_i$ by $i'$.
Finally, traverse the unit circle $\Ss^1$ counter-clockwise,
picking up the labels as you read, to obtain the desired word $w(X)$.
Notice that $m_Y(X) > 0$ if and only if $m_Y(w(X)) > 0$;
in particular, $w(X) \in \fW^{+}$, as desired.

\medskip\noindent
\begin{remark} 
 {\rm The above discussion implies that  the cyclic word $w(M)$ corresponding to any totally positive $2\times n$-matrix $M$  coincides with the cyclic word given by $(1,2,\dots,  n,  1', 2', \dots,  n')$. 
 We will call this particular cyclic word \emph{totally positive}. 
 The cyclic word corresponding to a totally negative matrix is obtained from the totally positive word by interchanging its every even entry with the opposite
and reading it backwards.
 We will call this cyclic word \emph{totally negative}.}
 \end{remark}
 \begin{example} {The cyclic word 
 $(12341'2'3'4')$ is totally positive while the cyclic word 
 $(43'21'4'32'1)$ is totally negative.}
 \end{example}
 





In all figures in the remaining part of  this section  the cyclic words should be read counter-clockwise along the circle. %

Given $w_0 \in \fW^{+}$,
let $\C[w_0] \subset \C_3$ be the set of matrices
$X \in \C_3$ for which $w(X) = w_0$.
Notice that $\C_3 = \bigsqcup_{w \in \fW^{+}} \C[w]$;
the following lemma shows that these subsets
are all well-behaved.

\begin{lemma}
\label{lm:contract} 
Given $w_0 \in \fW^{+}$,
the set $\C[w_0]$ is contractible (and nonempty).
\end{lemma}

\begin{proof}
We proceed by induction on $n$; the cases $n \le 3$ are easy.
In this proof, we write $\fW_n^{+}$ in order to avoid confusion.

Given a word $w_0 \in \fW_n^{+}$, let $w_1 \in \fW_{n-1}^{+}$
be obtained by removing $1$ and $1'$ from $w_0$ and then
subtracting $1$ from each remaining label
(thus, for instance, if $w_0 = 13'452'1'34'5'2$
then $w_1 = 2'34'1'23'4'1$).
Similarly, given $X_0 = (v_1, v_2, \ldots, v_n) \in \Lo^1_{(2\times n)}$ 
we obtain $X_1 = (v_2, \ldots, v_n) \in \Lo^1_{(2\times (n-1))}$
by removing the first column.

By induction hypothesis, $\C[w_1]$ is contractible.
Given $X_1 \in \C[w_1]$, the set of vectors $v_1 \in \RR^2$
which can be placed at the left of $X_1$ to obtain $X_0 \in \C[w_0]$
is a convex cone.
Thus $\C[w_0]$ is also contractible, as desired.
\end{proof}

Given a flag-convex curve $\Gamma_2: I \to \C_2$,
we are now interested in following the sequence of words $w$
corresponding to the sets $\C[w]$ traversed by $\Gamma_2$.
We first present a combinatorial description.
Let us now define \emph{admissible moves}
on the set of all admissible cyclic words.
(Below $\dots$ stands for an arbitrary sequence of labels
in  an admissible word.)

\smallskip
\begin{definition}\label{def:admissible}
	{  
	Below, we denote by $\alpha$ either the label $j$ or its opposite label $j'$. If $\alpha=j$, then $-\alpha=j'$, if $\alpha=j'$ then $-\alpha=j$.  
	For every $k>2$, the following two    moves are called  {\it admissible}: 
	
	\medskip
	\noindent 1) clockwise rotation of label $(k-1)$ toward point $k$\newline
	\noindent
	\small{
		$\dots,  k', \dots,  k-1, j, \dots,  k, \dots  (k-1)', j',\dots \to    \dots, k',\dots,  j, k-1, \dots, k, \dots j',(k-1)',\dots$}\\
\noindent 2) counter-clockwise rotation of label $(k-1)$ toward label $k$\newline
	\small{$\dots,  k, \dots, j, k-1, \dots,  k', \dots,  j',(k-1)',\dots \to  \dots,  k, \dots,  k-1,  j, \dots,  k', \dots (k-1)',j',\dots $
	}

The first admissible move describes the change of a cyclic word when the label $(k-1)$ rotates clockwise toward the label $k$ 
and passes the position of the label $j$ (or $j'$) while the second  admissible move describes the similar change when the label $(k-1)$ rotates counterclockwise. 
	
	\begin{figure}[H]
		\centering
		\resizebox{\MyFig}{!}{
			\begin{tikzpicture}[line cap=round,line join=round,>=triangle 45,x=1.0cm,y=1.0cm]
			\clip(0,0) rectangle (10.,3.5);
			
			\draw(2.5,2.) circle (1cm);
			\node[circle,fill=black,inner sep=0pt,minimum size=3pt,label=above:{\small{$k'$}}] (a) at (2.5 cm,3 cm) {};
			\node[circle,fill=black,inner sep=0pt,minimum size=3pt,label=below:{\small{$k$}}] (b) at (2.5 cm,1 cm) {};
			\draw[dashed] (a) - -  (b);
			\node[circle,fill=black,inner sep=0pt,minimum size=3pt,label=below right:{\small{$(k-1)'$}}] (c) at (3.37 cm,1.5 cm) {O};
			\node[circle,fill=black,inner sep=0pt,minimum size=3pt,label=above left:{\small{$k-1$}}] (d) at (1.63 cm,2.5 cm) {O};
			\draw[red,dashed] (c) - -  (d);
			\node[circle,fill=black,inner sep=0pt,minimum size=3pt,label=right:{\small{$-\alpha$}}] (e) at (3.5 cm,2 cm) {};
			\node[circle,fill=black,inner sep=0pt,minimum size=3pt,label=left:{\small{$\alpha$}}] (f) at (1.5 cm,2 cm) {};
			\draw[dashed] (e) - -  (f);
			
			\draw(7.5,2.) circle (1cm);
			\node[circle,fill=black,inner sep=0pt,minimum size=3pt,label=above:{\small{$k'$}}] (k) at (7.5 cm,3 cm) {};
			\node[circle,fill=black,inner sep=0pt,minimum size=3pt,label=below:{\small{$k$}}] (l) at (7.5 cm,1 cm) {};
			\draw[dashed] (k) - -  (l);
			\node[circle,fill=black,inner sep=0pt,minimum size=3pt,label=above right:{\small{$(k-1)'$}}] (m) at (8.37 cm,2.5 cm) {O};
			\node[circle,fill=black,inner sep=0pt,minimum size=3pt,label=below left:{\small{$k-1$}}] (n) at (6.63 cm,1.5 cm) {O};
			\draw[red,dashed] (m) - -  (n);
			\node[circle,fill=black,inner sep=0pt,minimum size=3pt,label=right:{\small{$-\alpha$}}] (o) at (8.5 cm,2 cm) {};
			\node[circle,fill=black,inner sep=0pt,minimum size=3pt,label=left:{\small{$\alpha$}}] (p) at (6.5 cm,2 cm) {};
			\draw[dashed] (o) - -  (p);
			
			\end{tikzpicture}
		}
		\caption{The first admissible move. The  label $k-1$ rotates counter-clockwise towards the label $k$.}
	\end{figure}
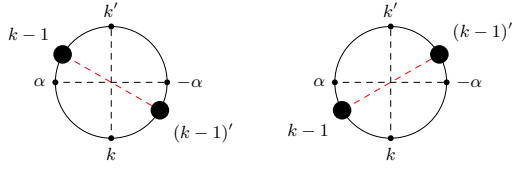

\vskip-0.7cm	
		
	\begin{figure}[H]
		\centering
		\resizebox{\MyFig}{!}{
			\begin{tikzpicture}[line cap=round,line join=round,>=triangle 45,x=1.0cm,y=1.0cm]
			\clip(0,0) rectangle (10.,3.5);
			
			\draw(2.5,2.) circle (1cm);
			\node[circle,fill=black,inner sep=0pt,minimum size=3pt,label=above:{\small{$k$}}] (a) at (2.5 cm,3 cm) {};
			\node[circle,fill=black,inner sep=0pt,minimum size=3pt,label=below:{\small{$k'$}}] (b) at (2.5 cm,1 cm) {};
			\draw[dashed] (a) - -  (b);
			\node[circle,fill=black,inner sep=0pt,minimum size=3pt,label=above right:{\small{$(k-1)'$}}] (c) at (3.37 cm,2.5 cm) {O};
			\node[circle,fill=black,inner sep=0pt,minimum size=3pt,label=below left:{\small{$k-1$}}] (d) at (1.63 cm,1.5 cm) {O};
			\draw[red,dashed] (c) - -  (d);
			\node[circle,fill=black,inner sep=0pt,minimum size=3pt,label=right:{\small{$-\alpha$}}] (e) at (3.5 cm,2 cm) {};
			\node[circle,fill=black,inner sep=0pt,minimum size=3pt,label=left:{\small{$\alpha$}}] (f) at (1.5 cm,2 cm) {};
			\draw[dashed] (e) - -  (f);
			
			\draw(7.5,2.) circle (1cm);
			\node[circle,fill=black,inner sep=0pt,minimum size=3pt,label=above:{\small{$k$}}] (k) at (7.5 cm,3 cm) {};
			\node[circle,fill=black,inner sep=0pt,minimum size=3pt,label=below:{\small{$k'$}}] (l) at (7.5 cm,1 cm) {};
			\draw[dashed] (k) - -  (l);
			\node[circle,fill=black,inner sep=0pt,minimum size=3pt,label=below right:{\small{$(k-1)'$}}] (m) at (8.37 cm,1.5 cm) {O};
			\node[circle,fill=black,inner sep=0pt,minimum size=3pt,label=above left:{\small{$k-1$}}] (n) at (6.63 cm,2.5 cm) {O};
			\draw[red,dashed] (m) - -  (n);
			\node[circle,fill=black,inner sep=0pt,minimum size=3pt,label=right:{\small{$-\alpha$}}] (o) at (8.5 cm,2 cm) {};
			\node[circle,fill=black,inner sep=0pt,minimum size=3pt,label=left:{\small{$\alpha$}}] (p) at (6.5 cm,2 cm) {};
			\draw[dashed] (o) - -  (p);
			
			\end{tikzpicture}
		}
		\caption{The second admissible move. The label $k-1$ rotates clockwise towards the label $k$.}
	\end{figure}
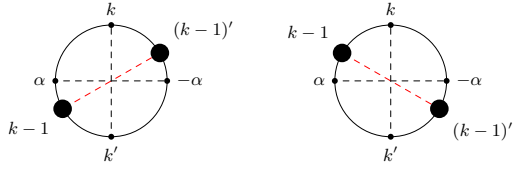
	}	
\end{definition}



For any flag-convex curve $\Gamma_2: I \to \C_2$
there are finitely many $t \in I$ for which $\Gamma_2(t) \notin \C_3$.
Indeed, a flag-convex curve transversally intersects
any hypersurface which is a connected component 
of $\C_2 \smallsetminus \C_3$;
a stronger result is proved in \cite{GoSa0}.

\begin{lemma}
\label{lemma:admissible}
Given a flag-convex curve $\Gamma_2: I \to \C_2$,
consider the sequence of words $w \in \fW^{+}$ for which
$\Gamma_2$ traverses $\C[w]$.
Then this sequence of words consists of admissible moves. 
\end{lemma}

\begin{proof}
By Lemma~\ref{lem:char},  
the tangent vector to the curve $\Gamma_2$ at $t_0$
belongs to the cone spanned by the vectors $\Gamma_2(t_0) \frakl_j$,
where $\frakl_j$ is the $n\times n$ matrix whose only nonzero entry
is located at the position $(j+1,j)$ and is equal to $1$. 
Note that  the right multiplication of an arbitrary $k\times n$-matrix
by $(\operatorname{id} + s\,\frakl_j)$ acts as a column operation
adding $s$ times the $j+1$st column  to the $j$th column. 
This corresponds to moving (infinitesimally) the  point labelled $j$
towards the point labelled  $j+1$ 
along the shortest of the two arcs of $\Ss^1$ connecting them.
Since any infinitesimal motion of the point configuration
induced by the curve  $\Gamma_2$ 
is represented as a positive linear combination
of such infinitesimal elementary moves
we can approximate the whole time evolution
of the point configuration
as a sequence of consecutive elementary moves
described in Definition~\ref{def:admissible}. 
\end{proof}
 
\begin{example}
\label{ex:L0N}{
Set
\[ \Gamma(t) = L_0 \exp(t N), \quad
L_0 = \begin{pmatrix} 1 & 0 & 0 & 0 \\ 0 & 1 & 0 & 0 \\
1/6 & 0 & 1 & 0 \\ 1/8 & 1/5 & 0 & 0 \end{pmatrix},
\quad
N = \begin{pmatrix} 0 & 0 & 0 & 0 \\ 1 & 0 & 0 & 0 \\
0 & 1 & 0 & 0 \\ 0 & 0 & 1 & 0 \end{pmatrix}. \]
The curve $\Gamma: \RR \to \Lo_4^1$ is flag-convex.
A simple computation verifies that the flag-convex curve $\Gamma_2: \RR \to \C$
is of the form $\Gamma_2: \RR \to \C_2$.
Indeed, the values of $t$ for which $m_Y(\Gamma_2(t)) = 0$ for some $Y$ are:
$t_1 \approx -0.63$ (for $Y = \{2,3\}$);
$t_2 = 0$ (for $Y = \{2,4\}$);
$t_3 \approx 0.26$ (for $Y = \{1,2\}$);
$t_4 \approx 0.63$ (for $Y = \{2,3\}$);
$t_5 \approx 0.77$ (for $Y = \{1,3\}$);
$t_6 \approx 1.11$ (for $Y = \{1,2\}$).
The corresponding sequence of words is:
$w_0 = 143'21'4'32'$,
$w_1 = 1423'1'4'2'3$,
$w_2 = 31243'1'2'4'$,
$w_3 = 32143'2'1'4'$,
$w_4 = 23142'3'1'4'$,
$w_5 = 21342'1'3'4'$,
$w_6 = 12341'2'3'4'$;
moves are admissible, as expected.}
\end{example}

 As discussed above,
the set of all admissible words labels
the set of all connected components of $\C_3$ (or of $\Omega_n$).
The connected components are separated by codimension 1 walls in
$\C_2 \smallsetminus \C_3$.
Admissible moves correspond to crossing walls between connected components
following flag-convex curves. 

\begin{remark}
We explained above that   when a (locally) convex curve
intersects the divisor $H_L$ where
$L=\operatorname{span}\langle e_1,\dots,e_{n-2}\rangle \subset \bR^n$,  
the respective admissible configuration of labelled points on $\Ss^1$
is acted upon by an admissible move
which either interchanges the relative order of the  points
labelled $1$  and $2$ or  the points $1$ and $2'$. 
\end{remark}




For an admissible cyclic word $w$ and any  of its two distinct entries 
$a$ and $b$ (belonging to  $\{1,\dots,n,1',\dots,n'\}$),
we denote by 
$\llbracket a,b \rrbracket$ the shortest closed arc in $\Ss^1$
starting at the point labeled $a$ and ending at the point labeled $b$.
In other words,   of two possible arcs connecting $a$ and $b$,
we choose the one whose length does not exceed half a turn
(or $(n-1)$ positions in the word $w$).

We call an arc $\llbracket a,b \rrbracket$
\emph{(increasing) decreasing} if
$a, b \in \{1, \ldots, n\}$,
$a<b$ and
$b$ is obtained from $a$ by a rotation
of less that a half-turn in the (counter-)clockwise direction.
Given  $i<j$, we say that an admissible cyclic word $w$
\emph{contains a monotone subsequence}
$\llbracket i\ldots j\rrbracket = \llbracket i,i+1,i+2,\dots,j\rrbracket$  
if, for $i\le k\le j-1$, 
all of the arcs $\llbracket k,k+1\rrbracket$  
are either simultaneously increasing or simultaneously decreasing. 
A monotone sequence $\llbracket i\ldots j\rrbracket$ can be interpreted
as an immersed arc $I \to \Ss^1$, $I = [i,j]$,
taking $i \in I$ to the point of $\Ss^1$ labeled $i$ and
taking $j \in I$ to the point labeled $j$;
notice that, unlike the subarcs $\llbracket k, k+1\rrbracket$,
such an immersed arc can be (much) longer than a half-turn.
The  \emph{content} of an immersed arc in $\Ss^1$
is the number of complete half-turns contained in the arc;
we denote the content of a monotone sequence
$\llbracket i\ldots j\rrbracket$
by $\Cont(\llbracket i\ldots j\rrbracket)$.

We call a monotone subsequence $\llbracket i\ldots j\rrbracket$ \emph{maximal}
if neither $\llbracket i-1\ldots j\rrbracket$ for $i>1$
nor $\llbracket i\ldots j+1\rrbracket$ for $j<n$ is monotone.
An admissible word $w$ can be interpreted as the concatenation
of its maximal monotone arcs
$\llbracket 1\ldots k_1 \rrbracket$,
$\llbracket k_1\ldots k_2 \rrbracket$, \dots,
$\llbracket k_{s-1} \ldots n \rrbracket$;
here $k_0 = 1$, $k_s = n$ and
$s$ is the total number of maximal monotone subsequences in $w$.
The \emph{total content} $\Cont(w)$ of an admissible word $w$ is
$\Cont(w) = \Cont(\llbracket 1\ldots k_1 \rrbracket) + \cdots +
\Cont(\llbracket k_{s-1} \ldots n \rrbracket)$,
the sum of the contents of all maximal monotone subsequences of $w$. 

\begin{definition}\label{def:rank}
We define $\rk(w)$, the \emph{rank} of the admissible word $w$, by
\[ \rk(w) = 2 \Cont(w) + s - 1. \]
For any matrix $X\in\C_3$, we define its rank $\rk(X)=\rk(w(X))$.
\end{definition}

%
%


\begin{example} {Consider the following cyclic words: 
\begin{enumerate}
\item for $w_1=123451'2'3'4'5'$,  we get $s=1, \Cont(w)=0$, and $\rk(w)=0$; 
\item for  $w_2=15'43'21'54'32'$,  we get $s=1, \Cont(w)=3$, and $\rk(w)=6$; 
\item in  $w_3=145231'4'5'2'3'$  there are $s = 3$ maximal monotone subsequences:
$\llbracket 1,2,3\rrbracket$,
$\llbracket 3,4\rrbracket$ and $\llbracket 4,5\rrbracket$.
Hence, $\Cont(w)=0$, and $\rk(w)=2$;
\item  for  $w_4=415234'1'5'2'3'$, one gets  $s=3$, $\Cont(w)=0$ and $\rk(w)=2$.
\end{enumerate}
}
\end{example}
\begin{remark} {\rm The word $w_1$ is totally positive,
while $w_2$ is  totally negative.
The word $w_4$ is obtained from $w_3$ by one admissible move
which shifts  $1$ closer to $2$;
$w_3$ can be obtained from $w_4$ by an admissible move
which shifts $4$ closer to $5$. 
Observe that $\rk(w_3)=\rk(w_4)$. }
\end{remark}

\begin{lemma}
\label{lm:rankstep}
Fix $n$ so that admissible words $w \in \fW^{+}$ have length $2n$.
For the totally positive word $w_{+} \in \fW^{+}$,
we have $\rk(w_{+}) = 0$.
For the totally negative word $w_{-} \in \fW^{+}$,
we have $\rk(w_{-}) = 2(n-2)$.
For any other admissible cyclic word
$w \in \fW^{+} \smallsetminus \{w_{+},w_{-}\}$, we have
$0 < \rk(w) < 2(n-2)$.
Furthermore, $m_{\{1,2\}}(w) > 0$ if and only if $\rk(w)$ is even.
\end{lemma}

\begin{proof}
The proof is by induction on $n$; the cases $n \le 3$ are easy.

Given $w_0 \in \fW_n^{+}$, let $w_1 \in \fW_{n-1}^{+}$
be obtained by removing $1$ and $1'$ and by decreasing by $1$
the remaining labels (as in the proof of Lemma \ref{lm:contract}).
Let $\llbracket 1\ldots k_1\rrbracket$ be the first
maximal monotonic subarc for $w_0$.
Notice that $k_1 = 2$ if and only if $m_{\{1,2\}}(w_0) \ne m_{\{2,3\}}(w_0)$
(recall that $m_Y(w) \in \{\pm 1\}$).
On the other hand, if $m_{\{1,2\}}(w_0) = m_{\{2,3\}}(w_0)$
then the first maximal monotonic arc for the word $w_1$ is 
$\llbracket 1\ldots (k_1-1)\rrbracket$.
Let $c_0 = \Cont(\llbracket 1\ldots k_1\rrbracket)$
and $c_1 = \Cont(\llbracket 2\ldots k_1\rrbracket)$
(both for $w_0$).
Notice that
$c_1 = \Cont(\llbracket 1\ldots (k_1-1)\rrbracket)$ for $w_1$.
We have either $c_0 = c_1$
(if, for $w_0$,
the points labeled $k_1$ and $k_1'$ are both outside the arc
$\llbracket 1, 2\rrbracket$)
or $c_0 = c_1 + 1$ (otherwise).
We thus have
\[ \rk(w_0) = \rk(w_1) +
\begin{cases}
0, & m_{\{1,2\}}(w_0) = m_{\{2,3\}}(w_0), c_0 = c_1, \\
1, & m_{\{1,2\}}(w_0) \ne m_{\{2,3\}}(w_0), \\
2, & m_{\{1,2\}}(w_0) = m_{\{2,3\}}(w_0), c_0 = c_1 + 1;
\end{cases} \]
this provides us with the desired induction step.
\end{proof}

The next statement is the most important technical step
in our proof of Theorem~\ref{th:main}.
The argument is simple but a little long, and is done case by case;
it is presented and illustrated by the series of ten figures shown below. 

\medskip
\noindent
\begin{proposition}\label{prop:rank_move}
Consider $w_0, w_1 \in \fW_n^{+}$.
Assume that an admissible move takes $w_0$ to $w_1$:
then $\rk(w_1) \le \rk(w_0)$.
Furthermore, if $m_{\{1,2\}}(w_1) \ne m_{\{1,2\}}(w_0)$
then $\rk(w_1) < \rk(w_0)$.
\end{proposition}

Notice that the last claim follows from the first claim
together with the parity remark in Lemma \ref{lm:rankstep}.


\setstcolor{red}

\begin{proof}

Below we present all possible  types of elementary moves and,  for each of them,  we analyze what happens with the rank of $w$. 
Observe that during the evolution of the point configuration in $(\Ss^1)^n$
following some curve $\Gamma_2$
the rank of configuration  does not change until two or more configuration points collide. 
We consider admissible moves and list all cases when a moving point labelled $i$ collides with one of the remaining points labelled $j$ or $j^\prime$.
Detailed consideration of all possible cases led us to their subdivision into the following types of collisions. (This subdivision is an artifact of our proof). 

\medskip
\noindent
Type Ia:  the moving point $1$ collides with the  point $k$, $k>2$;

\noindent
Type Ib: the moving point $1$ collides with the point $k'$, $k>2$; 

\noindent
Type IIa: the moving point $i$, $i>1$ collides with the point $1$; 

\noindent
Type IIb:  the moving point $i$, $i>1$ collides with the point $1'$;
 
\noindent
Type IIIa: the moving point $i$ collides with $j$ when both $i,j>1$, $j\ne i-1$; 

If $j=i-1$  the case needs to be subdivided into two subcases by the location of  point $i-2$  in one of the following two intervals:
 
 \noindent
Type IIIb: the moving point $i$ collides with $j=i-1$, $i>1$ and the point $i-2$ belongs to the shortest arc between $(i-1)'$ and $i$; 

\noindent
 Type IIIc: the moving point $i$ collides with $j=i-1$, $i>1$ and the point $i-2$ belongs to the shortest arc between $i-1$ and $i'$. 

\noindent
 Type IVa: the moving point $i$ collides with $j'$, $i,j>1$, $j\ne i-1$; 

If $j=i-1$ then case  needs to be subdivided into two subcases also by the location of point $i-2$  in one of the following two intervals:
 
  \noindent
Type IVb: the moving point $i$ collides with $j=i-1$, $i>1$ and the point $i-2$ belongs to the shortest arc between $(i-1)$ and $i$;
 
  \noindent
 Type IVc: the moving point $i$ collides with $j=i-1$, $i>1$ and the point $i-2$ belongs to the shortest arc between $(i-1)'$ and $i'$.  

The above types exhaust all possible situations of collision and we discuss below what happens with the rank function under these collisions. 

\medskip
\definecolor{xdxdff}{rgb}{0.49019607843137253,0.49019607843137253,1.}
\definecolor{zzttqq}{rgb}{0.6,0.2,0.}
\definecolor{qqqqff}{rgb}{0.,0.,1.}
\definecolor{uuuuuu}{rgb}{0.26666666666666666,0.26666666666666666,0.26666666666666666}

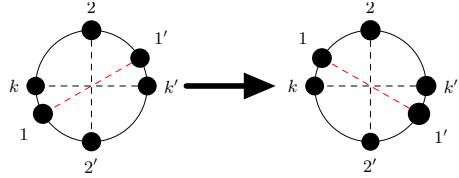
\begin{figure}[H]
	\centering
	\resizebox{\MyFig}{!}{
\begin{tikzpicture}[line cap=round,line join=round,>=triangle 45,x=1.0cm,y=1.0cm]
\clip(0,0) rectangle (10.,3.5);
\draw(2.5,2.) circle (1cm);
\node[circle,fill=black,inner sep=0pt,minimum size=3pt,label=above:{\small{$2$}}] (a) at (2.5 cm,3 cm) {A};
\node[circle,fill=black,inner sep=0pt,minimum size=3pt,label=below:{\small{$2'$}}] (b) at (2.5 cm,1 cm) {B};
\draw[dashed] (a) - -  (b);
\node[circle,fill=black,inner sep=0pt,minimum size=3pt,label=above right:{\small{$1'$}}] (c) at (3.37 cm,2.5 cm) {C};
\node[circle,fill=black,inner sep=0pt,minimum size=3pt,label=below left:{\small{$1$}}] (d) at (1.63 cm,1.5 cm) {D};
\draw[red,dashed] (c) - -  (d);
\node[circle,fill=black,inner sep=0pt,minimum size=3pt,label=right:{\small{$k'$}}] (e) at (3.5 cm,2 cm) {E};
\node[circle,fill=black,inner sep=0pt,minimum size=3pt,label=left:{\small{$k$}}] (f) at (1.5 cm,2 cm) {F};
\draw[dashed] (e) - -  (f);
\draw(7.5,2.) circle (1cm);
\node[circle,fill=black,inner sep=0pt,minimum size=3pt,label=above:{\small{$2$}}] (k) at (7.5 cm,3 cm) {K};
\node[circle,fill=black,inner sep=0pt,minimum size=3pt,label=below:{\small{$2'$}}] (l) at (7.5 cm,1 cm) {L};
\draw[dashed] (k) - -  (l);
\node[circle,fill=black,inner sep=0pt,minimum size=3pt,label=below right:{\small{$1'$}}] (m) at (8.37 cm,1.5 cm) {M};
\node[circle,fill=black,inner sep=0pt,minimum size=3pt,label=above left:{\small{$1$}}] (n) at (6.63 cm,2.5 cm) {N};
\draw[red,dashed] (m) - -  (n);
\node[circle,fill=black,inner sep=0pt,minimum size=3pt,label=right:{\small{$k'$}}] (o) at (8.5 cm,2 cm) {O};
\node[circle,fill=black,inner sep=0pt,minimum size=3pt,label=left:{\small{$k$}}] (p) at (6.5 cm,2 cm) {P};
\draw[dashed] (o) - -  (p);
\draw[->, line width=1mm](4.2 cm, 2 cm) -- (5.8 cm, 2 cm);
\end{tikzpicture}
}
\caption{Elementary admissible move of type Ia with $k>2$.}
\end{figure}

\vskip-0.3cm
{\small  The move changes the relative order of  points $1$ and $k$ in cyclic word $w$. Rank $\rk(w)$ 
 does not change.}

\begin{figure}[H]
	\centering
	\resizebox{\MyFig}{!}{
\begin{tikzpicture}[line cap=round,line join=round,>=triangle 45,x=1.0cm,y=1.0cm]
\clip(0,0) rectangle (10.,3.5);

\draw(2.5,2.) circle (1cm);
\node[circle,fill=black,inner sep=0pt,minimum size=3pt,label=above:{\small{$2$}}] (a) at (2.5 cm,3 cm) {A};
\node[circle,fill=black,inner sep=0pt,minimum size=3pt,label=below:{\small{$2'$}}] (b) at (2.5 cm,1 cm) {B};
\draw[dashed] (a) - -  (b);
\node[circle,fill=black,inner sep=0pt,minimum size=3pt,label=above right:{\small{$1'$}}] (c) at (3.37 cm,2.5 cm) {C};
\node[circle,fill=black,inner sep=0pt,minimum size=3pt,label=below left:{\small{$1$}}] (d) at (1.63 cm,1.5 cm) {D};
\draw[red,dashed] (c) - -  (d);
\node[circle,fill=black,inner sep=0pt,minimum size=3pt,label=right:{\small{$k$}}] (e) at (3.5 cm,2 cm) {E};
\node[circle,fill=black,inner sep=0pt,minimum size=3pt,label=left:{\small{$k'$}}] (f) at (1.5 cm,2 cm) {F};
\draw[dashed] (e) - -  (f);

\draw(7.5,2.) circle (1cm);
\node[circle,fill=black,inner sep=0pt,minimum size=3pt,label=above:{\small{$2$}}] (k) at (7.5 cm,3 cm) {K};
\node[circle,fill=black,inner sep=0pt,minimum size=3pt,label=below:{\small{$2'$}}] (l) at (7.5 cm,1 cm) {L};
\draw[dashed] (k) - -  (l);
\node[circle,fill=black,inner sep=0pt,minimum size=3pt,label=below right:{\small{$1'$}}] (m) at (8.37 cm,1.5 cm) {M};
\node[circle,fill=black,inner sep=0pt,minimum size=3pt,label=above left:{\small{$1$}}] (n) at (6.63 cm,2.5 cm) {N};
\draw[red,dashed] (m) - -  (n);
\node[circle,fill=black,inner sep=0pt,minimum size=3pt,label=right:{\small{$k$}}] (o) at (8.5 cm,2 cm) {O};
\node[circle,fill=black,inner sep=0pt,minimum size=3pt,label=left:{\small{$k'$}}] (p) at (6.5 cm,2 cm) {P};
\draw[dashed] (o) - -  (p);

\draw[->, line width=1mm](4.2 cm, 2 cm) -- (5.8 cm, 2 cm);

\end{tikzpicture}
}
\caption{Elementary admissible move of type Ib. }
\end{figure}
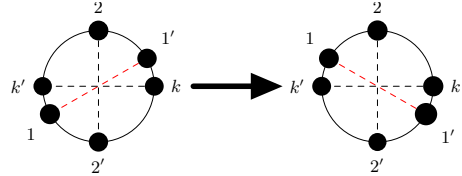
\vskip - 0.3cm
{\small The move changes the relative order of points $1$ and $k'$ in the word $w$. If the maximal element $k_1$ of the first maximal monotone subsequence $[1,2,\dots,k_1]$ is different from $k$,  then $\rk(w)$ does not change. If $k_1=k$, then $\rk(w)$ decreases by $1$.}


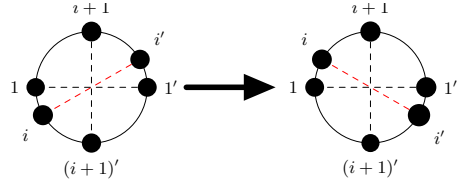
\begin{figure}[H]
	\centering
	\resizebox{\MyFig}{!}{
\begin{tikzpicture}[line cap=round,line join=round,>=triangle 45,x=1.0cm,y=1.0cm]
\clip(0,0) rectangle (10.,3.5);

\draw(2.5,2.) circle (1cm);
\node[circle,fill=black,inner sep=0pt,minimum size=3pt,label=above:{\small{$i+1$}}] (a) at (2.5 cm,3 cm) {A};
\node[circle,fill=black,inner sep=0pt,minimum size=3pt,label=below:{\small{$(i+1)'$}}] (b) at (2.5 cm,1 cm) {B};
\draw[dashed] (a) - -  (b);
\node[circle,fill=black,inner sep=0pt,minimum size=3pt,label=above right:{\small{$i'$}}] (c) at (3.37 cm,2.5 cm) {C};
\node[circle,fill=black,inner sep=0pt,minimum size=3pt,label=below left:{\small{$i$}}] (d) at (1.63 cm,1.5 cm) {D};
\draw[red,dashed] (c) - -  (d);
\node[circle,fill=black,inner sep=0pt,minimum size=3pt,label=right:{\small{$1'$}}] (e) at (3.5 cm,2 cm) {E};
\node[circle,fill=black,inner sep=0pt,minimum size=3pt,label=left:{\small{$1$}}] (f) at (1.5 cm,2 cm) {F};
\draw[dashed] (e) - -  (f);

\draw(7.5,2.) circle (1cm);
\node[circle,fill=black,inner sep=0pt,minimum size=3pt,label=above:{\small{$i+1$}}] (k) at (7.5 cm,3 cm) {K};
\node[circle,fill=black,inner sep=0pt,minimum size=3pt,label=below:{\small{$(i+1)'$}}] (l) at (7.5 cm,1 cm) {L};
\draw[dashed] (k) - -  (l);
\node[circle,fill=black,inner sep=0pt,minimum size=3pt,label=below right:{\small{$i'$}}] (m) at (8.37 cm,1.5 cm) {M};
\node[circle,fill=black,inner sep=0pt,minimum size=3pt,label=above left:{\small{$i$}}] (n) at (6.63 cm,2.5 cm) {N};
\draw[red,dashed] (m) - -  (n);
\node[circle,fill=black,inner sep=0pt,minimum size=3pt,label=right:{\small{$1'$}}] (o) at (8.5 cm,2 cm) {O};
\node[circle,fill=black,inner sep=0pt,minimum size=3pt,label=left:{\small{$1$}}] (p) at (6.5 cm,2 cm) {P};
\draw[dashed] (o) - -  (p);

\draw[->, line width=1mm](4.2 cm, 2 cm) -- (5.8 cm, 2 cm);

\end{tikzpicture}
}
\caption{Elementary admissible move of type IIa. }
\end{figure}
\vskip-0.3cm
{\small  If $i>2$ and the first monotone subsequence is ${\bf seq}_1=[1,2,\dots,k_1]$ with $k_1\ne i$, then $\rk(w)$ does not change.
If $i>2$ and $k_1=i$, then ${\bf seq}_1$ is increasing (since otherwise, $k_1\ge i+1$) and $\Cont(w)$ decreases by $1$, hence $\rk(w)$ drops by $2$.
Finally, if $i=2$, then $\rk(w)$ drops by $1$.}

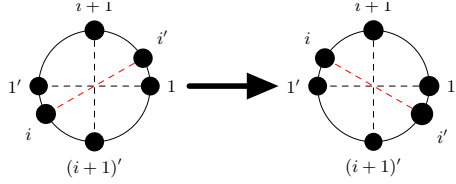
\begin{figure}[H]
	\centering
	\resizebox{\MyFig}{!}{
\begin{tikzpicture}[line cap=round,line join=round,>=triangle 45,x=1.0cm,y=1.0cm]
\clip(0,0) rectangle (10.,3.5);

\draw(2.5,2.) circle (1cm);
\node[circle,fill=black,inner sep=0pt,minimum size=3pt,label=above:{\small{$i+1$}}] (a) at (2.5 cm,3 cm) {A};
\node[circle,fill=black,inner sep=0pt,minimum size=3pt,label=below:{\small{$(i+1)'$}}] (b) at (2.5 cm,1 cm) {B};
\draw[dashed] (a) - -  (b);
\node[circle,fill=black,inner sep=0pt,minimum size=3pt,label=above right:{\small{$i'$}}] (c) at (3.37 cm,2.5 cm) {C};
\node[circle,fill=black,inner sep=0pt,minimum size=3pt,label=below left:{\small{$i$}}] (d) at (1.63 cm,1.5 cm) {D};
\draw[red,dashed] (c) - -  (d);
\node[circle,fill=black,inner sep=0pt,minimum size=3pt,label=right:{\small{$1$}}] (e) at (3.5 cm,2 cm) {E};
\node[circle,fill=black,inner sep=0pt,minimum size=3pt,label=left:{\small{$1'$}}] (f) at (1.5 cm,2 cm) {F};
\draw[dashed] (e) - -  (f);

\draw(7.5,2.) circle (1cm);
\node[circle,fill=black,inner sep=0pt,minimum size=3pt,label=above:{\small{$i+1$}}] (k) at (7.5 cm,3 cm) {K};
\node[circle,fill=black,inner sep=0pt,minimum size=3pt,label=below:{\small{$(i+1)'$}}] (l) at (7.5 cm,1 cm) {L};
\draw[dashed] (k) - -  (l);
\node[circle,fill=black,inner sep=0pt,minimum size=3pt,label=below right:{\small{$i'$}}] (m) at (8.37 cm,1.5 cm) {M};
\node[circle,fill=black,inner sep=0pt,minimum size=3pt,label=above left:{\small{$i$}}] (n) at (6.63 cm,2.5 cm) {N};
\draw[red,dashed] (m) - -  (n);
\node[circle,fill=black,inner sep=0pt,minimum size=3pt,label=right:{\small{$1$}}] (o) at (8.5 cm,2 cm) {O};
\node[circle,fill=black,inner sep=0pt,minimum size=3pt,label=left:{\small{$1'$}}] (p) at (6.5 cm,2 cm) {P};
\draw[dashed] (o) - -  (p);

\draw[->, line width=1mm](4.2 cm, 2 cm) -- (5.8 cm, 2 cm);
\end{tikzpicture}
}
\caption{Elementary admissible move of type IIb. }
\end{figure}
\vskip-0.3cm

{\small As in type III, if $i>2$ and the first monotone subsequence is ${\bf seq}_1=[1,2,\dots,k_1]$ with $k_1\ne i$, then $\rk(w)$ does not change.
If $i>2$ and $k_1=i$,  then ${\bf seq}_1$ is increasing (since otherwise, $k_1\ge i+1$) and $\rk(w)$ drops by $2$.
Finally, if $i=2$, then $\rk(w)$ drops by $1$.}

\vskip-0.3cm

\begin{figure}[H]
	\centering
	\resizebox{\MyFig}{!}{
\begin{tikzpicture}[line cap=round,line join=round,>=triangle 45,x=1.0cm,y=1.0cm]
\clip(0,0) rectangle (10.,3.5);

\draw(2.5,2.) circle (1cm);
\node[circle,fill=black,inner sep=0pt,minimum size=3pt,label=above:{\small{$i+1$}}] (a) at (2.5 cm,3 cm) {A};
\node[circle,fill=black,inner sep=0pt,minimum size=3pt,label=below:{\small{$(i+1)'$}}] (b) at (2.5 cm,1 cm) {B};
\draw[dashed] (a) - -  (b);
\node[circle,fill=black,inner sep=0pt,minimum size=3pt,label=above right:{\small{$i'$}}] (c) at (3.37 cm,2.5 cm) {C};
\node[circle,fill=black,inner sep=0pt,minimum size=3pt,label=below left:{\small{$i$}}] (d) at (1.63 cm,1.5 cm) {D};
\draw[red,dashed] (c) - -  (d);
\node[circle,fill=black,inner sep=0pt,minimum size=3pt,label=right:{\small{$j'$}}] (e) at (3.5 cm,2 cm) {E};
\node[circle,fill=black,inner sep=0pt,minimum size=3pt,label=left:{\small{$j$}}] (f) at (1.5 cm,2 cm) {F};
\draw[dashed] (e) - -  (f);

\draw(7.5,2.) circle (1cm);
\node[circle,fill=black,inner sep=0pt,minimum size=3pt,label=above:{\small{$i+1$}}] (k) at (7.5 cm,3 cm) {K};
\node[circle,fill=black,inner sep=0pt,minimum size=3pt,label=below:{\small{$(i+1)'$}}] (l) at (7.5 cm,1 cm) {L};
\draw[dashed] (k) - -  (l);
\node[circle,fill=black,inner sep=0pt,minimum size=3pt,label=below right:{\small{$i'$}}] (m) at (8.37 cm,1.5 cm) {M};
\node[circle,fill=black,inner sep=0pt,minimum size=3pt,label=above left:{\small{$i$}}] (n) at (6.63 cm,2.5 cm) {N};
\draw[red,dashed] (m) - -  (n);
\node[circle,fill=black,inner sep=0pt,minimum size=3pt,label=right:{\small{$j'$}}] (o) at (8.5 cm,2 cm) {O};
\node[circle,fill=black,inner sep=0pt,minimum size=3pt,label=left:{\small{$j$}}] (p) at (6.5 cm,2 cm) {P};
\draw[dashed] (o) - -  (p);

\draw[->, line width=1mm](4.2 cm, 2 cm) -- (5.8 cm, 2 cm);
\end{tikzpicture}
}
\caption{Elementary admissible move of type IIIa with $i,j>1$ and $j\notin \{i,i+1\}$.}
\end{figure}
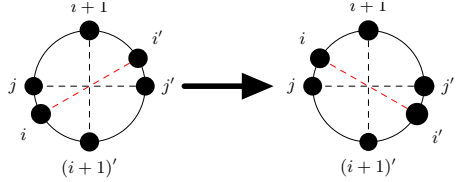
\vskip-0.5cm
{\small 
In this case $\rk(w)$ can only change  if  either $i<j$ and there exists an increasing maximal monotone subsequence $[i,i+1,\dots,j]$ or if $i>j$  and there exists a decreasing maximal monotone subsequence $[j,j+1,\dots,i]$. In both cases, $\Cont(w)$ decreases by $1$ and $\rk(w)$ decreases by $2$. The remaining situation  $j=i-1$ is considered in detail below.}

We split  the case $j=i-1$ in  Figure~7 into several subcases according to the relative position of $i-2$. The point $i-2$ can be located either in the interval $((i-1)',i)$  or in $(i-1,i')$, as  below. 

\vskip-0.3cm

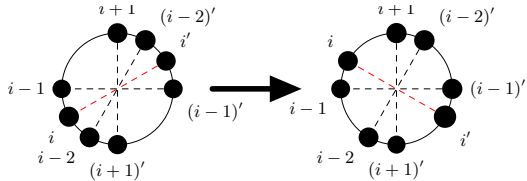
\begin{figure}[H]
	\centering
	\resizebox{\MyFig}{!}{
		\begin{tikzpicture}[line cap=round,line join=round,>=triangle 45,x=1.0cm,y=1.0cm]
		\clip(0,0) rectangle (10.,3.5);
		
		\draw(2.5,2.) circle (1cm);
		\node[circle,fill=black,inner sep=0pt,minimum size=3pt,label=above:{\small{$i+1$}}] (a) at (2.5 cm,3 cm) {A};
		\node[circle,fill=black,inner sep=0pt,minimum size=3pt,label=below:{\small{$(i+1)'$}}] (b) at (2.5 cm,1 cm) {B};
		\draw[dashed] (a) - -  (b);
		\node[circle,fill=black,inner sep=0pt,minimum size=3pt,label=above right:{\small{$i'$}}] (c) at (3.37 cm,2.5 cm) {C};
		\node[circle,fill=black,inner sep=0pt,minimum size=3pt,label=below left:{\small{$i$}}] (d) at (1.63 cm,1.5 cm) {D};
		\draw[red,dashed] (c) - -  (d);
		\node[circle,fill=black,inner sep=0pt,minimum size=3pt,label=below right:{\small{$(i-1)'$}}] (e) at (3.5 cm,2 cm) {E};
		\node[circle,fill=black,inner sep=0pt,minimum size=3pt,label=left:{\small{$i-1$}}] (f) at (1.5 cm,2 cm) {F};
		\draw[dashed] (e) - -  (f);
		\node[circle,fill=black,inner sep=0pt,minimum size=3pt,label=above right:{\small{$(i-2)'$}}] (g) at (3. cm,2.87 cm) {G};
        \node[circle,fill=black,inner sep=0pt,minimum size=3pt,label=below left:{\small{$i-2$}}] (h) at (2 cm,1.13 cm) {H};
        \draw[dashed] (g) - -  (h);
		
		\draw(7.5,2.) circle (1cm);
		\node[circle,fill=black,inner sep=0pt,minimum size=3pt,label=above:{\small{$i+1$}}] (k) at (7.5 cm,3 cm) {K};
		\node[circle,fill=black,inner sep=0pt,minimum size=3pt,label=below:{\small{$(i+1)'$}}] (l) at (7.5 cm,1 cm) {L};
		\draw[dashed] (k) - -  (l);
		\node[circle,fill=black,inner sep=0pt,minimum size=3pt,label=below right:{\small{$i'$}}] (m) at (8.37 cm,1.5 cm) {M};
		\node[circle,fill=black,inner sep=0pt,minimum size=3pt,label=above left:{\small{$i$}}] (n) at (6.63 cm,2.5 cm) {N};
		\draw[red,dashed] (m) - -  (n);
		\node[circle,fill=black,inner sep=0pt,minimum size=3pt,label=right:{\small{$(i-1)'$}}] (o) at (8.5 cm,2 cm) {O};
		\node[circle,fill=black,inner sep=0pt,minimum size=3pt,label=below left:{\small{$i-1$}}] (p) at (6.5 cm,2 cm) {P};
		\draw[dashed] (o) - -  (p);
		\node[circle,fill=black,inner sep=0pt,minimum size=3pt,label=above right:{\small{$(i-2)'$}}] (g) at (8. cm,2.87 cm) {G};
		\node[circle,fill=black,inner sep=0pt,minimum size=3pt,label=below left:{\small{$i-2$}}] (h) at (7 cm,1.13 cm) {H};
		\draw[dashed] (g) - -  (h);
\draw[->, line width=1mm](4.2 cm, 2 cm) -- (5.8 cm, 2 cm);	
		\end{tikzpicture}
	}
	\caption{Elementary admissible move of type IIIb. }
\end{figure}
\vskip-0.3cm
{\small The point $i-2$ belongs to  $((i-1)',i)$. $\Cont(w)$ does not change. Both the number of maximal monotone subsequences and $\rk(w)$ decrease by $2$.}

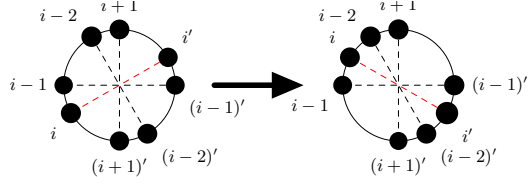
\begin{figure}[H]
	\centering
	\resizebox{\MyFig}{!}{
		\begin{tikzpicture}[line cap=round,line join=round,>=triangle 45,x=1.0cm,y=1.0cm]
		\clip(0,0) rectangle (10.,3.5);
		
		\draw(2.5,2.) circle (1cm);
		\node[circle,fill=black,inner sep=0pt,minimum size=3pt,label=above:{\small{$i+1$}}] (a) at (2.5 cm,3 cm) {A};
		\node[circle,fill=black,inner sep=0pt,minimum size=3pt,label=below:{\small{$(i+1)'$}}] (b) at (2.5 cm,1 cm) {B};
		\draw[dashed] (a) - -  (b);
		\node[circle,fill=black,inner sep=0pt,minimum size=3pt,label=above right:{\small{$i'$}}] (c) at (3.37 cm,2.5 cm) {C};
		\node[circle,fill=black,inner sep=0pt,minimum size=3pt,label=below left:{\small{$i$}}] (d) at (1.63 cm,1.5 cm) {D};
		\draw[red,dashed] (c) - -  (d);
		\node[circle,fill=black,inner sep=0pt,minimum size=3pt,label=below right:{\small{$(i-1)'$}}] (e) at (3.5 cm,2 cm) {E};
		\node[circle,fill=black,inner sep=0pt,minimum size=3pt,label=left:{\small{$i-1$}}] (f) at (1.5 cm,2 cm) {F};
		\draw[dashed] (e) - -  (f);
		\node[circle,fill=black,inner sep=0pt,minimum size=3pt,label=above left:{\small{$i-2$}}] (g) at (2. cm,2.87 cm) {G};
		\node[circle,fill=black,inner sep=0pt,minimum size=3pt,label=below right:{\small{$(i-2)'$}}] (h) at (3 cm,1.13 cm) {H};
		\draw[dashed] (g) - -  (h);
		
		\draw(7.5,2.) circle (1cm);
		\node[circle,fill=black,inner sep=0pt,minimum size=3pt,label=above:{\small{$i+1$}}] (k) at (7.5 cm,3 cm) {K};
		\node[circle,fill=black,inner sep=0pt,minimum size=3pt,label=below:{\small{$(i+1)'$}}] (l) at (7.5 cm,1 cm) {L};
		\draw[dashed] (k) - -  (l);
		\node[circle,fill=black,inner sep=0pt,minimum size=3pt,label=below right:{\small{$i'$}}] (m) at (8.37 cm,1.5 cm) {M};
		\node[circle,fill=black,inner sep=0pt,minimum size=3pt,label=above left:{\small{$i$}}] (n) at (6.63 cm,2.5 cm) {N};
		\draw[red,dashed] (m) - -  (n);
		\node[circle,fill=black,inner sep=0pt,minimum size=3pt,label=right:{\small{$(i-1)'$}}] (o) at (8.5 cm,2 cm) {O};
		\node[circle,fill=black,inner sep=0pt,minimum size=3pt,label=below left:{\small{$i-1$}}] (p) at (6.5 cm,2 cm) {P};
		\draw[dashed] (o) - -  (p);
		\node[circle,fill=black,inner sep=0pt,minimum size=3pt,label=above left:{\small{$i-2$}}] (g) at (7. cm,2.87 cm) {G};
		\node[circle,fill=black,inner sep=0pt,minimum size=3pt,label=below right:{\small{$(i-2)'$}}] (h) at (8 cm,1.13 cm) {H};
		\draw[dashed] (g) - -  (h);

\draw[->, line width=1mm](4.2 cm, 2 cm) -- (5.8 cm, 2 cm);		
		\end{tikzpicture}
	}
	\caption{Elementary admissible move of type IIIc.}
\end{figure}
\vskip-0.3cm
{\small The point $i-2$ belongs to $(i-1,i')$. The  rank $\rk(w)$ does not change.}



Finally, 
\vskip-0.3cm
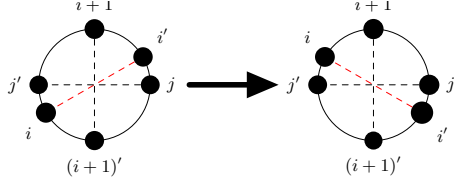
\begin{figure}[H]
	\centering
	\resizebox{\MyFig}{!}{
\begin{tikzpicture}[line cap=round,line join=round,>=triangle 45,x=1.0cm,y=1.0cm]
\clip(0,0) rectangle (10.,3.5);

\draw(2.5,2.) circle (1cm);
\node[circle,fill=black,inner sep=0pt,minimum size=3pt,label=above:{\small{$i+1$}}] (a) at (2.5 cm,3 cm) {A};
\node[circle,fill=black,inner sep=0pt,minimum size=3pt,label=below:{\small{$(i+1)'$}}] (b) at (2.5 cm,1 cm) {B};
\draw[dashed] (a) - -  (b);
\node[circle,fill=black,inner sep=0pt,minimum size=3pt,label=above right:{\small{$i'$}}] (c) at (3.37 cm,2.5 cm) {C};
\node[circle,fill=black,inner sep=0pt,minimum size=3pt,label=below left:{\small{$i$}}] (d) at (1.63 cm,1.5 cm) {D};
\draw[red,dashed] (c) - -  (d);
\node[circle,fill=black,inner sep=0pt,minimum size=3pt,label=right:{\small{$j$}}] (e) at (3.5 cm,2 cm) {E};
\node[circle,fill=black,inner sep=0pt,minimum size=3pt,label=left:{\small{$j'$}}] (f) at (1.5 cm,2 cm) {F};
\draw[dashed] (e) - -  (f);

\draw(7.5,2.) circle (1cm);
\node[circle,fill=black,inner sep=0pt,minimum size=3pt,label=above:{\small{$i+1$}}] (k) at (7.5 cm,3 cm) {K};
\node[circle,fill=black,inner sep=0pt,minimum size=3pt,label=below:{\small{$(i+1)'$}}] (l) at (7.5 cm,1 cm) {L};
\draw[dashed] (k) - -  (l);
\node[circle,fill=black,inner sep=0pt,minimum size=3pt,label=below right:{\small{$i'$}}] (m) at (8.37 cm,1.5 cm) {M};
\node[circle,fill=black,inner sep=0pt,minimum size=3pt,label=above left:{\small{$i$}}] (n) at (6.63 cm,2.5 cm) {N};
\draw[red,dashed] (m) - -  (n);
\node[circle,fill=black,inner sep=0pt,minimum size=3pt,label=right:{\small{$j$}}] (o) at (8.5 cm,2 cm) {O};
\node[circle,fill=black,inner sep=0pt,minimum size=3pt,label=left:{\small{$j'$}}] (p) at (6.5 cm,2 cm) {P};
\draw[dashed] (o) - -  (p);
\draw[->, line width=1mm](4.2 cm, 2 cm) -- (5.8 cm, 2 cm);

\end{tikzpicture}
}
\caption{Elementary admissible move of type IVa. }
\end{figure}
\vskip-0.3cm
{\small Here $i,j>1$, and $j$ is not in  $\{i,i+1\}$. If $j\ne i-1$, then $\rk(w)$ does not change unless either $i<j$ and  there exists a maximal decreasing subsequence $[i,i+1,\dots,j]$ or $i>j$ and there exists a maximal increasing subsequence $[j,j+1,\dots,i]$. In both cases $\Cont(w)$ decreases by $1$ and $\rk(w)$ decreases by $2$. 
The remaining sitiuation $j=i-1$ is considered in detail below.}

As above, we split  the case $j=i-1$ in the last figure into several subcases according to the relative position of $i-2$. The point $i-2$ can be located either in the interval $((i-1),i)$  or in the interval $((i-1)',i')$, as below.
\vskip-0.3cm

\begin{figure}[H]
	\centering
	\resizebox{\MyFig}{!}{
		\begin{tikzpicture}[line cap=round,line join=round,>=triangle 45,x=1.0cm,y=1.0cm]
		\clip(0,0) rectangle (10.,3.5);
		
		\draw(2.5,2.) circle (1cm);
		\node[circle,fill=black,inner sep=0pt,minimum size=3pt,label=above:{\small{$i+1$}}] (a) at (2.5 cm,3 cm) {A};
		\node[circle,fill=black,inner sep=0pt,minimum size=3pt,label=below:{\small{$(i+1)'$}}] (b) at (2.5 cm,1 cm) {B};
		\draw[dashed] (a) - -  (b);
		\node[circle,fill=black,inner sep=0pt,minimum size=3pt,label=above right:{\small{$i'$}}] (c) at (3.37 cm,2.5 cm) {C};
		\node[circle,fill=black,inner sep=0pt,minimum size=3pt,label=below left:{\small{$i$}}] (d) at (1.63 cm,1.5 cm) {D};
		\draw[red,dashed] (c) - -  (d);
		\node[circle,fill=black,inner sep=0pt,minimum size=3pt,label=below right:{\small{$i-1$}}] (e) at (3.5 cm,2 cm) {E};
		\node[circle,fill=black,inner sep=0pt,minimum size=3pt,label=left:{\small{$(i-1)'$}}] (f) at (1.5 cm,2 cm) {F};
		\draw[dashed] (e) - -  (f);
		\node[circle,fill=black,inner sep=0pt,minimum size=3pt,label=above right:{\small{$(i-2)'$}}] (g) at (3. cm,2.87 cm) {G};
        \node[circle,fill=black,inner sep=0pt,minimum size=3pt,label=below left:{\small{$i-2$}}] (h) at (2 cm,1.13 cm) {H};
        \draw[dashed] (g) - -  (h);
		
		\draw(7.5,2.) circle (1cm);
		\node[circle,fill=black,inner sep=0pt,minimum size=3pt,label=above:{\small{$i+1$}}] (k) at (7.5 cm,3 cm) {K};
		\node[circle,fill=black,inner sep=0pt,minimum size=3pt,label=below:{\small{$(i+1)'$}}] (l) at (7.5 cm,1 cm) {L};
		\draw[dashed] (k) - -  (l);
		\node[circle,fill=black,inner sep=0pt,minimum size=3pt,label=below right:{\small{$i'$}}] (m) at (8.37 cm,1.5 cm) {M};
		\node[circle,fill=black,inner sep=0pt,minimum size=3pt,label=above left:{\small{$i$}}] (n) at (6.63 cm,2.5 cm) {N};
		\draw[red,dashed] (m) - -  (n);
		\node[circle,fill=black,inner sep=0pt,minimum size=3pt,label=right:{\small{$i-1$}}] (o) at (8.5 cm,2 cm) {O};
		\node[circle,fill=black,inner sep=0pt,minimum size=3pt,label=below left:{\small{$(i-1)'$}}] (p) at (6.5 cm,2 cm) {P};
		\draw[dashed] (o) - -  (p);
		\node[circle,fill=black,inner sep=0pt,minimum size=3pt,label=above right:{\small{$(i-2)'$}}] (g) at (8. cm,2.87 cm) {G};
		\node[circle,fill=black,inner sep=0pt,minimum size=3pt,label=below left:{\small{$i-2$}}] (h) at (7 cm,1.13 cm) {H};
		\draw[dashed] (g) - -  (h);
\draw[->, line width=1mm](4.2 cm, 2 cm) -- (5.8 cm, 2 cm);	
		\end{tikzpicture}
	}
	\caption{Elementary admissible move of type IVb.}
\end{figure}
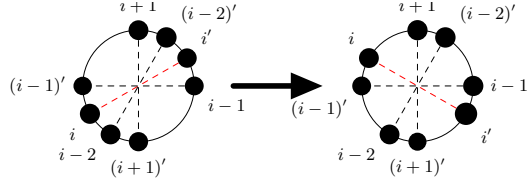
\vskip-0.5cm
{\small The point $i-2$ belongs to $(i-1,i)$. $\Cont(w)$ and the number $s$ of maximal monotone subsequences do not change. Hence, $\rk(w)$ does not change either.}
\vskip-0.3cm
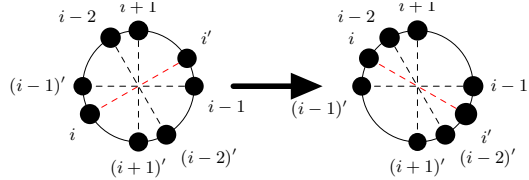
\begin{figure}[H]
	\centering
	\resizebox{\MyFig}{!}{
		\begin{tikzpicture}[line cap=round,line join=round,>=triangle 45,x=1.0cm,y=1.0cm]
		\clip(0,0) rectangle (10.,3.5);
		
		\draw(2.5,2.) circle (1cm);
		\node[circle,fill=black,inner sep=0pt,minimum size=3pt,label=above:{\small{$i+1$}}] (a) at (2.5 cm,3 cm) {A};
		\node[circle,fill=black,inner sep=0pt,minimum size=3pt,label=below:{\small{$(i+1)'$}}] (b) at (2.5 cm,1 cm) {B};
		\draw[dashed] (a) - -  (b);
		\node[circle,fill=black,inner sep=0pt,minimum size=3pt,label=above right:{\small{$i'$}}] (c) at (3.37 cm,2.5 cm) {C};
		\node[circle,fill=black,inner sep=0pt,minimum size=3pt,label=below left:{\small{$i$}}] (d) at (1.63 cm,1.5 cm) {D};
		\draw[red,dashed] (c) - -  (d);
		\node[circle,fill=black,inner sep=0pt,minimum size=3pt,label=below right:{\small{$i-1$}}] (e) at (3.5 cm,2 cm) {E};
		\node[circle,fill=black,inner sep=0pt,minimum size=3pt,label=left:{\small{$(i-1)'$}}] (f) at (1.5 cm,2 cm) {F};
		\draw[dashed] (e) - -  (f);
		\node[circle,fill=black,inner sep=0pt,minimum size=3pt,label=above left:{\small{$i-2$}}] (g) at (2. cm,2.87 cm) {G};
		\node[circle,fill=black,inner sep=0pt,minimum size=3pt,label=below right:{\small{$(i-2)'$}}] (h) at (3 cm,1.13 cm) {H};
		\draw[dashed] (g) - -  (h);
		
		\draw(7.5,2.) circle (1cm);
		\node[circle,fill=black,inner sep=0pt,minimum size=3pt,label=above:{\small{$i+1$}}] (k) at (7.5 cm,3 cm) {K};
		\node[circle,fill=black,inner sep=0pt,minimum size=3pt,label=below:{\small{$(i+1)'$}}] (l) at (7.5 cm,1 cm) {L};
		\draw[dashed] (k) - -  (l);
		\node[circle,fill=black,inner sep=0pt,minimum size=3pt,label=below right:{\small{$i'$}}] (m) at (8.37 cm,1.5 cm) {M};
		\node[circle,fill=black,inner sep=0pt,minimum size=3pt,label=above left:{\small{$i$}}] (n) at (6.63 cm,2.5 cm) {N};
		\draw[red,dashed] (m) - -  (n);
		\node[circle,fill=black,inner sep=0pt,minimum size=3pt,label=right:{\small{$i-1$}}] (o) at (8.5 cm,2 cm) {O};
		\node[circle,fill=black,inner sep=0pt,minimum size=3pt,label=below left:{\small{$(i-1)'$}}] (p) at (6.5 cm,2 cm) {P};
		\draw[dashed] (o) - -  (p);
		\node[circle,fill=black,inner sep=0pt,minimum size=3pt,label=above left:{\small{$i-2$}}] (g) at (7. cm,2.87 cm) {G};
		\node[circle,fill=black,inner sep=0pt,minimum size=3pt,label=below right:{\small{$(i-2)'$}}] (h) at (8 cm,1.13 cm) {H};
		\draw[dashed] (g) - -  (h);

\draw[->, line width=1mm](4.2 cm, 2 cm) -- (5.8 cm, 2 cm);		
		\end{tikzpicture}
	}
	\caption{Elementary admissible move of type IVc.}
\end{figure}
\vskip-0.3cm
{\small The point $i-2$ belongs to $((i-1)',i')$.
	Either $\rk(w)$ does not change or it decreases by $2$.}

\medskip
We have analyzed all the possible types of  admissible elementary moves and concluded that, for each admissible move which changes the sign of $m_{\{1,2\}}$, $\rk(w)$ decreases. \end{proof}

\medskip


Recall the unipotent matrix $N$ in Example \ref{ex:L0N};
notice that
\[ (\exp(tN))_{i,j} = \begin{cases}
0, & i<j; \\
t^{i-j}/(i-j)!, & i\ge j. 
\end{cases} \]

\begin{lemma}\label{lem:NegPos}
For any $n\times n$-matrix $G\in \Lo_n^1$, 
\begin{enumerate}
\item there exists $t_{+}>0$ such that  $G \exp(tN)$
is totally positive for any $t> t_{+}$;
\item there exists $t_{-}<0$ such that  $G \exp(tN)$
is totally negative for any $t < t_{-}$.
\end{enumerate}
\end{lemma}

\begin{proof}
Write $T(t) = \exp(tN)$.
Note that $G_{ij}=1$ if $i=j$ and $G_{ij} = 0$ if $i<j$.
Hence, 
$$(G T(t))_{ij}=T_{ij}(t)+\text{lower order terms in  }  t.$$  Any minor $m_T$ of $T(t)$ equals $m_T(t)=a t^d$ for some positive $a$ and $d$.  The corresponding minor $m_G$
of $G\cdot T(t)$ equals $m_G(t)=a t^d+p_{m,G}(t)$, where $p_{m,G}$ is a polynomial of degree strictly less than $d$.
Hence, for $t$  such that  $|t|\gg 0$ the sign of $m_G(t)$ coincides with that of $m_T(t)$.   It remains to notice that $T(t)$ is totally positive for positive $t$ and totally negative for negative $t$ and the lemma follows.
\end{proof}

\begin{corollary}\label{cor:extension}
Let $\gamma:I\to \bP^{n-1}$ be a globally convex curve
and $\gamma_\Ff:I\to \Lo_n^1$ be its osculating flag curve
(considered in the appropriate open Schubert cell identified with $\Lo_n^1$).
Then, $\gamma$ can be extended to a globally convex
$\tilde\gamma:[a,b]\to \bP^{n-1}$,
$I\subset [a,b]$ such that
$\tilde\gamma_\Ff:[a,b]\to   \Lo_n^1$,
$\tilde\gamma_\Ff(a)\in \Neg$,
$\tilde\gamma_\Ff(b)\in \Pos$.
\end{corollary}

\begin{proof}   Take $I=[s,f]$ and  set $b=f+t_+(\gamma_\Ff(f))$,
$a=s+t_-(\gamma_\Ff(s))$. Define  
$$\tilde\gamma_\Ff(t)=
\begin{cases}
\gamma_\Ff(t)\ , &  \text{ for } t\in I;\\
\gamma_\Ff(f)\cdot T(t-f)\ , & \text{ for } t\in [f,b];\\
\gamma_\Ff(s)\cdot T(t-s)\ , & \text{ for } t\in [a,s],
\end{cases}
$$ where $T(t)$ is defined in Lemma~\ref{lem:NegPos}.

We define $\tilde\gamma(t)$ by taking the first column of the matrix $\tilde\gamma_\Ff(t)$.
Lemma~\ref{lem:char}  implies that the curve $\tilde\gamma(t)$ is 
globally convex.


Finally, by definition of $t_-$ and $t_+$, one has that  $\tilde\gamma_\Ff(a)\in \Neg$,  $\tilde\gamma_\Ff(b)\in \Pos$.

Notice that the curve $\tilde\gamma_{\Ff}$ is usually not smooth.
The curve can be perturbed to become smooth
or we can work with a larger space of curves.
These issues are amply discussed in several papers,
including \cite{GoSaX, GoSa1}.
\end{proof}

\noindent
\begin{proof}[Proof of Theorem~\ref{th:main}] 
Consider a flag-convex curve $\Gamma_2: [a,b] \to \cC_2$.
From the previous results we may assume that
$\Gamma_2(a) \in \Neg$, $\Gamma_2(b) \in \Pos$.
Consider the associated word as a function of $t$:
as $t$ increases, we perform admissible moves.
The rank of the associated word starts
as $2(n-2)$ and ends as $0$.
The rank never increases and decreases by $1$
at every zero of $m_2$;
it may also decrease at other points.
Thus, the number of zeros of $m_2$ is at most $2(n-2)$,
as desired.
\end{proof}

\bigskip\goodbreak

\section{Proof of Theorem \ref{th:La}}
\label{sec:La}

We start by introducing a special class of matrix curves.
Namely, given a pair of matrices $(N_0,L_0)$,
where $N_0$ is a nilpotent lower triangular matrix
with positive subdiagonal entries and zero entries elsewhere,
and $L_0 \in \Lo_n^1$,
define the curve  $\Gamma_{N_0,L_0}: \RR \to \Lo_n^1$ as given by
\begin{equation}\label{eq:curve}
\Gamma_{N_0,L_0}(t) = L_0 \exp(t N_0).
\end{equation}
One can easily see that $\Gamma_{N_0,L_0}$ is  flag-convex
and, for $i > j$,
its entry $(i,j)$ is a polynomial of degree $i - j$.
We call such flag-convex curves \emph{polynomial}.
(They are closely related to the fundamental solutions
of the simplest differential equation $y^{(n)}=0$.) 



For a polynomial curve $\Gamma_{N_0,L_0}$, the function 
$m_k(t)$ is indeed a real polynomial of degree $k(n-k)$ in $t$.  
So for polynomial curves, Conjecture~\ref{conj:Wr} trivially holds.

\smallskip
In this section we will first prove Proposition \ref{prop:L},
a warm-up result, and then prove Theorem~\ref{th:La}.
Proposition~\ref{prop:L} shows that there exist polynomial flag-convex curves
which are non-transversal to the reference flag at $(n^3-n)/6$ distinct points
which implies that the estimates in Theorems \ref{th:La}
(and therefore also in Corollary \ref{coro:ineq})
hold for polynomial curves. 


\begin{proposition}
\label{prop:L}
Choose  a nilpotent lower triangular matrix  $N_0$ 
with positive subdiagonal entries and zero entries elsewhere. 
Then there exists $L_0  \in \Lo_n^1$ such that,
for the polynomial curve $\Ga_{N_0,L_0}$
given by Equation \eqref{eq:curve} and every  $k=1,\dots, n-1$,
all roots of $m_k(t)$ are real and simple.
Furthermore, $L_0$ can be taken so that all such roots are distinct, 
implying that there are  totally exactly $(n^3 - n)/6$ such roots,
all real and distinct.
\end{proposition}

To settle Proposition~\ref{prop:L} we need more notation.
As in \cite{GoSa0} (see especially Sections 2 and 7),
let $S_n$ be the symmetric group with generators $a_i = (i\quad i+1)$.
The symmetric group is endowed with the usual Bruhat order.
The top permutation (or the Coxeter element) of $S_n$
is denoted by $\eta$ (another common notation is $w_0$).
For a permutation $\sigma \in S_n$,
define its \emph{multiplicity vector} $\mult(\sigma) \in \ZZ^{n-1}$
with coordinates
$\mult_k(\sigma) = (1^\sigma + \cdots + k^\sigma) - (1 + \cdots + k)$;
thus, $\mult_k(\eta) = k(n-k)$.
If $\sigma_0 \triangleleft \sigma_1 = \sigma_0 a_j = (i_0 i_1) \sigma_0$,
then 
\begin{equation}
\label{equation:lemma24}
\mult_k(\sigma_1) = \begin{cases}
\mult_k(\sigma_0) + 1, & i_0 \le k < i_1, \\
\mult_k(\sigma_0), & \textrm{otherwise;} \end{cases} 
\end{equation}
this is Lemma 2.4 in \cite{GoSa0}.

For $\rho \in S_n$, the permutation matrix $P_\rho$ has nonzero entries
in positions $(i,i^\rho)$ so that
$e_i^\transpose P_\rho = e_{i^\rho}^\transpose$.
Apply the Bruhat factorization to decompose $\Lo_n^1$
as a disjoint union of subsets $\Bru_\rho$, $\rho \in S_n$.
More precisely, for $L \in \Lo_n^1$, write $L \in \Bru_\rho$
if and only if there exist upper triangular matrices $U_1$ and $U_2$
such that $L = U_1 P_\rho U_2$.
In particular, $\Bru_{e} = \{I\}$ and
$\Bru_{\eta}$ is open and dense.
If $\Gamma: I \to \Lo_{n}^1$ is smooth and flag-convex,
$\Gamma(0) \in \Bru_{\rho}$ and $\sigma = \eta\rho$ then
$t = 0$ is a root of multiplicity $\mult_k(\sigma)$ of $m_k(t) = 0$;
this is Theorem 4 in \cite{GoSa0}.

Recall that $\frakl_j$ is the matrix whose only nonzero entry equals $1$
in position $(j+1,j)$.
Let $\lambda_j(t) = \exp(t \frakl_j) \in \Lo_n^1$ so that
$\lambda_j(t)$ has an entry equal to $t$ in position $(j+1,j)$;
the remaining entries equal $1$ (on the main diagonal)
and $0$ (elsewhere).
If $\rho_0 \triangleleft \rho_1 = \rho_0 a_j$,
$L_0 \in \Bru_{\rho_0}$ and $t \ne 0$, then
$L_1 = L_0 \lambda_j(t) \in \Bru_{\rho_1}$
(see Section 5 in \cite{GoSa0}).

Consider $N_0$ arbitrary but fixed,
as in the statement of Proposition \ref{prop:L}.
Given $L_0 \in \Lo_n^1$, construct the curve $\Gamma_{N_0,L_0}(t) = L_0 \exp(t N_0)$
and the real polynomials $m_k(t) \in \RR[t]$ as above.
We say that a matrix $L_0$ is \emph{$\rho$-good} if and only if
$L_0 \in \Bru_{\rho} \subset \Lo_n^1$ and,
for all $k$, all nonzero roots of $m_k$ are real and simple.
Notice that $\IdI$ is (vacuously) $e$-good.

\begin{lemma}
\label{lemma:goodstep}
Consider $\rho_0, \rho_1 \in S_n$,
$\rho_0 \triangleleft \rho_1 = \rho_0 a_j$.
Let $L_0 \in \Bru_{\rho_0} \subset \Lo_n^1$ be a $\rho_0$-good matrix.
Then there exists $\epsilon > 0$ such that,
for all $\tau \in \RR$ satisfying
the restriction $0 < |\tau| < \epsilon$ one has that  
$L_1 = L_0 \lambda_j(\tau)$ is $\rho_1$-good.
\end{lemma}

\begin{proof}
As above, we have $L_1 \in \Bru_{\rho_1}$.
For $\tau$ near $0$, nonzero real simple roots of $m_k$
remain nonzero, real and simple.

Let $\sigma_0 = \eta\rho_0$ and $\sigma_1 = \eta\rho_1$ so that
$\sigma_1 \triangleleft \sigma_0 = \sigma_1 a_j = (i_0 i_1) \sigma_1$.
As in Equation \ref{equation:lemma24},
$\mult_k(\sigma_1) = \mult_k(\sigma_0) - 1$ for $i_0 \le k < i_1$
and
$\mult_k(\sigma_1) = \mult_k(\sigma_0)$ otherwise.
Originally (i.e. for $L_0$) the root $t = 0$  has multiplicity
$\mult_k(\sigma_0)$;
after perturbation (i.e. for $L_1$) it has multiplicity
$\mult_k(\sigma_1)$.
Thus, for $k < i_0$ or $k \ge i_1$ no new root is born
and we are done.
For $i_0 \le k < i_1$ exactly one new root is born:
it must therefore be real and, for small $|\tau|$, simple.
\end{proof}

\begin{lemma}
\label{lemma:goodmatrix}
For all $\rho \in S_n$ there exist $\rho$-good matrices.
\end{lemma}

\begin{proof}
Consider a reduced word $\rho = a_{i_1} \cdots a_{i_l}$
where $l = \inv(\rho)$ is the number of inversions of $\rho$. 
For $k \le l$, define $\rho_k = a_{i_1} \cdots a_{i_k}$;
in particular, $\rho_0 = e$ and $\rho_l = \rho$.
As mentioned above, $I$ is $\rho_0$-good.
Apply Lemma \ref{lemma:goodstep} to deduce that
if there exists a $\rho_k$-good matrix then
there exists a $\rho_{k+1}$-good matrix.
The result follows by induction.
\end{proof}

\begin{proof}[Proof of Proposition \ref{prop:L}]
By Lemma \ref{lemma:goodmatrix},
there exists an $\eta$-good matrix $L_0$.
The roots of every polynomial $m_k(t)$ are real and simple.
The same holds for any $\tilde L_0 \in A$
where $A$ is some sufficiently small open neighborhood of $L_0$.
It suffices to show that for some such $\tilde L_0$
all roots are distinct.

Let $\rho \in S_n$ be different from $\eta$ and $\eta a_i$, $1 \le i < n$.
Then $\Bru_\rho \subset \Lo_n^1$ is a submanifold of codimension at least 2.
Define
\[ X_\rho = \{ L \exp(t N_0); L \in \Bru_\rho, t \in \RR\}, \qquad
Y = \Lo_n^1 \smallsetminus
\bigcup_{\rho \in S_n \smallsetminus \{\eta, \eta a_1, \ldots, \eta a_{n-1}\}}
X_\rho; \]
each set $X_\rho$ has measure zero.
The set $Y$ has total measure and is therefore dense.
Take $\tilde L_0 \in A \cap Y$.
We claim that all roots of the polynomials $m_k$ are real, simple and distinct,
as desired.

Indeed, assume by contradiction that
$m_{k_1}(t_0) = m_{k_2}(t_0) = 0$, $k_1 < k_2$.
Take $\rho \in S_n$ such that
$\Gamma_{N_0,L_0}(t_0) = \tilde L_0 \exp(t_0 N_0) \in \Bru_\rho$;
set $\sigma = \eta\rho$.
We have that $\mult_{k_1}(\sigma) \ge 1$ and $\mult_{k_2}(\sigma) \ge 1$
 whence $\sigma \notin \{e,a_1,\ldots,a_{n-1}\}$
and therefore
$\rho \in S_n \smallsetminus \{\eta, \eta a_1, \ldots, \eta a_{n-1}\}$.
Thus $\tilde L_0 = \Gamma_{N_0,L_0}(t_0) \exp(-t_0 N_0) \in X_\rho$
and therefore $\tilde L_0 \notin Y$, a contradiction.
\end{proof}

\begin{example}
\label{example:abacdcbacb}
For $n = 5$, let $N_0$ be the matrix with subdiagonal entries equal to $1$.
Write $\eta = a_1 a_2 a_1 a_3 a_4 a_3 a_2 a_1 a_3 a_2$,
an arbitrary reduced word.
The matrices 
\[ \lambda_1(1), \quad
\lambda_1(1) \lambda_2(-1), \quad
\lambda_1(1) \lambda_2(-1) \lambda_1(1) \]
are easily seen to be $a_1$-, $a_1a_2$- and $a_1a_2a_1$-good, respectively
(signs are chosen in an arbitrary manner).
If we thus proceed from left to right,
at each step taking a number of sufficiently small absolute value,
we obtain the following example of an $(\eta b)$-good matrix:
\[ L_0 = \lambda_1(1) \lambda_2(-1) \lambda_1(1) 
\lambda_3(\frac18) \lambda_4(-\frac18) \lambda_3(\frac{1}{64})
\lambda_2(\frac{1}{512}) \lambda_1(-\frac{1}{512}) \lambda_3(-\frac{1}{4096}).
\]
For $\Gamma_{N_0,L_0}(t) = L_0 \exp(t N_0)$,
all roots of the polynomials $m_k$ are real, simple and distinct
(see also Remark \ref{remark:itinerary}).
\end{example}



Let $\Pos \subset \Lo_n^1$ (resp. $\Neg \subset \Lo_n^1$)
be the open subset of totally positive (resp. negative) matrices.
If $\Gamma: I \to \Lo_n^1$ is flag-convex and
$\Gamma(t_0) \in \overline\Pos$ then
$\Gamma(t) \in \Pos$ for all $t > t_0$;
similarly, if $\Gamma(t_0) \in \overline\Neg$ then
$\Gamma(t) \in \Neg$ for all $t < t_0$:
see Lemma \ref{lem:ShGoSa0} above
and Lemma 5.7 from \cite{GoSa0}.

We are almost ready to prove Theorem~\ref{th:La}. 
We may assume without loss of generality
that $0 \in I_1 \subset I$
and that $\Gamma_\bullet(0) = \IdI$.
Take $I_2 = [t_{-},t_{+}] \subseteq I_1$ with
$t_{-} < 0 < t_{+}$.
Notice that $t = 0$ is the only root of $m_{\Gamma_\bullet;k}$ in $I$:
also, for $t > 0$ we have $\Gamma(t) \in \Pos$
and for $t < 0$ we have $\Gamma(t) \in \Neg$.
Also, $t = 0$ is a root of multiplicity $k(n-k) = \mult_k(\eta)$
of $m_{\Gamma_\bullet;k}$.
Given $L_1 \in \Lo_n^1$, set $\Gamma_1(t) = L_1 \Gamma_\bullet(t)$;
write $m_{L_1;k} = m_{\Gamma_1;k}$.
Thus, if $L_1 \in \Bru_\rho$ and $\sigma = \eta\rho$ then
$t = 0$ is a root of multiplicity $\mult_k(\sigma)$ of $m_{L_1;k}$.

For $\rho = \eta\sigma \in S_n$, a matrix $L_1$ is \emph{$\rho$-good}
(for $\Gamma_\bullet$ and $I_2 = [t_{-},t_{+}]$ fixed)
if and only if, for $\Gamma_1(t) = L_1 \Gamma_\bullet(t)$, we have:
$L_1 \in \Bru_{\rho} \subset \Lo_n^1$,
$\Gamma_1(t_{+}) \in \Pos$, $\Gamma_1(t_{-}) \in \Neg$ and,
for all $k$, the function $m_{L_1;k}$ admits precisely
$k(n-k) - \mult_k(\sigma)$ nonzero roots in $I \smallsetminus \{0\}$,
all in the interior of $I_2$ and all simple.
Recall that $L_1 \in \Bru_{\rho}$ 
implies that $t = 0$ is a root of $m_{L_1,k}$
of multiplicity $\mult_k(\sigma)$.
Notice that $\IdI$ is $e$-good.

\begin{lemma}
\label{lemma:goodstepLa}
Consider $\rho_0, \rho_1 \in S_n$,
$\rho_0 \triangleleft \rho_1 = \rho_0 a_j$.
Let $L_0 \in \Bru_{\rho_0} \subset \Lo_n^1$ be a $\rho_0$-good matrix.
Then there exists $\epsilon > 0$ such that,
for all $\tau \in \RR$,
if $0 < |\tau| < \epsilon$, then 
$L_\tau = L_0 \lambda_j(\tau)$ is $\rho_1$-good.
\end{lemma}

\begin{proof}
Let $\sigma_0 = \eta\rho_0$ and $\sigma_1 = \eta\rho_1$ so that
$\sigma_1 \triangleleft \sigma_0 = \sigma_1 a_j = (i_0 i_1) \sigma_1$.
As above, we have $L_\tau \in \Bru_{\rho_1}$ for $\tau \ne 0$.
Write $\mu_k = \mult_k(\sigma_0)$.
As in Equation \ref{equation:lemma24},
$\mult_k(\sigma_1) = \mu_k - 1$ for $i_0 \le k < i_1$, 
and
$\mult_k(\sigma_1) = \mu_k$ otherwise.

Since $\Pos$ and $\Neg$ are open sets,
for $\tau$ near $0$ we have
$\Gamma_{\tau}(t_{+}) \in \Pos$ and
$\Gamma_{\tau}(t_{-}) \in \Neg$
(where $\Gamma_{\tau}(t) = L_{\tau} \Gamma_\bullet(t)$).
This condition will be assumed from now on.

For $\tau$ near $0$, the $k(n-k) - \mult_k(\sigma_0)$
nonzero simple roots of $m_{L_\tau,k}$
remain nonzero and simple;
from the previous paragraph,
they are in the interior of $I_2$.
By compactness, for small $|\tau|$,
there are no new roots away from a small neighborhood of $t = 0$.

The root $t = 0$ has multiplicity $\mu_k$ for $m_{L_0;k}$.
Let $s_k \in \{\pm 1\}$ be the sign of $m_{L_0;k}^{(\mu_k)}(0) \ne 0$
so that $s_k m_{L_0;k}^{(\mu_k)}(t) > 0$
in a small neighborhood $I_0 \subset I_2$ of $t = 0$.
For small $|\tau|$, we likewise have
$s_k m_{L_\tau;k}^{(\mu_k)}(t) > 0$ in $I_0$.
For $\tau \ne 0$, the root $t = 0$ has multiplicity
$\mult_k(\sigma_1)$.
For $k < i_0$ or $k \ge i_1$, we have $\mult_k(\sigma_1) = \mu_k$, 
and therefore $t = 0$ is the only root in $I_0$ and we are done.
For $i_0 \le k < i_1$ we have $\mult_k(\sigma_1) = \mu_k - 1$.
The signs of $m_k$ at the extrema of $I_0$
together with the sign of $m_k^{(\mu_k)}$
and the multiplicity of the zero at $t=0$ imply that,
for small $|\tau|$, there is exactly one new nonzero root
of $m_{L_\tau;k}$ in $I_0 \smallsetminus \{0\}$;
this root is simple, as desired.
\end{proof}

\begin{lemma}
\label{lemma:goodmatrixLa}
Consider a flag-convex curve $\Gamma_\bullet: I \to \Lo_n^1$
and intervals $I_2 \subseteq I$ fixed, as above.
For all $\rho \in S_n$ there exist $\rho$-good matrices.
\end{lemma}

\begin{proof}
Consider a reduced word $\rho = a_{i_1} \cdots a_{i_l}$
where $l = \inv(\rho)$ is the number of inversions of $\rho$. 
For $k \le l$ define $\rho_k = a_{i_1} \cdots a_{i_k}$;
in particular, $\rho_0 = e$ and $\rho_l = \rho$.
As remarked, $Id$ is $\rho_0$-good.
Apply Lemma \ref{lemma:goodstepLa} to deduce that
if there exists a $\rho_k$-good matrix, then
there exists a $\rho_{k+1}$-good matrix.
The result follows by induction.
\end{proof}

\begin{proof}[Proof of Theorem \ref{th:La}]
For $L \in \Lo_n^1$,
set $\Gamma_{L}(t) = L \Gamma_0(t)$
and $m_{L;k} = m_{\Gamma_{L};k}$.
By Lemma \ref{lemma:goodmatrixLa},
there exists an $\eta$-good matrix $L_0$.
Each function $m_{L_0;k}$ has exactly $k(n-k)$ roots in
the interior of $I_2$, all simple.
Since $\Gamma_{L_0}(t_{+}) \in \Pos$ and $\Gamma_{L_0}(t_{-}) \in \Neg$,
there are no other roots.
These conditions are open and therefore hold
for the functions $m_{\tilde L_0;k}$ for any $\tilde L_0 \in A$
where $A$ is some sufficiently small open neighborhood of $L_0$.
It suffices to show that, for some such $\tilde L_0$, 
all roots are distinct.

Consider the quotient map $\Lo_n^1 \to \RR$
taking $L$ to $L_{2,1} + L_{3,2} + \cdots + L_{n,n-1}$;
pre-images of points form a family of parallel hyperplanes.
Let $S \subset A$, $L_0 \in S$,
be a convex neighborhood of $L_0$ in its hyperplane;
notice that $S$ is transversal to $\Gamma'_{L_0}(0)$.
The function $\Phi: S \times I_1 \to \Lo_n^1$ defined by
$\Phi(L,t) = \Gamma_L(t)$ is
a tubular neighborhood of the image $\Gamma_{L_0}[I_1]$.
This may require replacing the original interval $I_1$
by a smaller interval, still with $0$ in the interior:
notice that this is allowed.

Let $\rho \in S_n$ be different from $\eta$ and $\eta a_i$, $1 \le i < n$.
Then $\Bru_\rho \subset \Lo_n^1$ is a submanifold of codimension at least 2,
and therefore so is $\Phi^{-1}[\Bru_\rho] \subset S \times I_1$.
Let $X_\rho \subset S$ be its image under the projection onto $S$:
the subset $X_\rho \subset S$ has measure zero.
Let
\[ Y = S \smallsetminus
\bigcup_{\rho \in S_n \smallsetminus \{\eta, \eta a_1, \ldots, \eta a_{n-1}\}}
X_\rho: \]
the subset $Y \subseteq S$ has total measure and is therefore dense.
Notice that since $Y \subseteq S \subset A$,
if $\tilde L_0 \in Y$, then
the function $m_{\tilde L_0;k}$ has precisely $k(n-k)$ roots in $I_1$,
all simple and all in the interior of $I_1$.
We claim that in this case
all roots of the functions $m_{\tilde L_0;k}$ are also distinct,
as desired.

Indeed, assume by contradiction that
$m_{k_1}(t_0) = m_{k_2}(t_0)$, $k_1 < k_2$.
Take $\rho \in S_n$ such that
$\Gamma_{\tilde L_0}(t_0) = \Phi(\tilde L_0, t_0) \in \Bru_\rho$;
set $\sigma = \eta\rho$.
We have $\mult_{k_1}(\sigma) \ge 1$ and $\mult_{k_2}(\sigma) \ge 1$
whence $\sigma \notin \{e,a_1,\ldots,a_{n-1}\}$;
thus,
$\rho \in S_n \smallsetminus \{\eta, \eta a_1, \ldots, \eta a_{n-1}\}$.
Thus $(\tilde L_0, t_0) \in \Phi^{-1}[\Bru_\rho]$
and therefore $\tilde L_0 \in X_\rho$
and therefore $\tilde L_0 \notin Y$, a contradiction.
\end{proof}

\begin{remark}
\label{remark:itinerary}
In \cite{GoSa1}, we introduce the concept of {\em itinerary}
$\iti(\Gamma)$ of a locally convex curve $\Gamma$.
The itinerary is a word with letters in $S_n \smallsetminus \{e\}$
which gives important information about the curve:
essentially, it lists the moments of non-transversality
according to Bruhat cell.
The construction applies to flag-convex curves $\Gamma: I \to \Lo_n^1$.
In the proof of Theorem \ref{th:La},
we start with a curve $\Gamma_{\bullet}$ with itinerary
$\iti(\Gamma_{\bullet}) = \eta$ (a word with a single letter).
We construct a small perturbation $\Gamma_1$ of the original curve.
The itinerary is now a word of length $\frac{n^3 - n}{6}$,
each letter being a generator $a_k$ of $S_n$.
Each letter $a_k$ appears $k(n-k)$ times.
There are many different reduced words for $\eta$
and even for a given word we may have more that one itinerary.
The itinerary of the curve $\Gamma_{N_0,L_0}$
constructed in Example \ref{example:abacdcbacb}
is $\iti(\Gamma_{N_0,L_0}) = dcbcdabacbcbabdcbadc$.
Conjecture \ref{conj:Wr} is equivalent to the statement
that for any flag-convex curve $\Gamma: I \to \Lo_n^1$
the letter $a_k$ appears at most $k(n-k)$ times in the itinerary
$\iti(\Gamma)$.
\end{remark}

\end{document}